\documentclass[11pt]{article}
\usepackage{amssymb, amsmath, amsthm, srcltx}
\usepackage{hyperref}
\usepackage{graphicx}
\usepackage{color}
\newtheorem{thm}{Theorem}[section]
\newtheorem{lem}{Lemma}[section]

\theoremstyle{definition}

\theoremstyle{remark}
\newtheorem{rem}[thm]{Remark}
\numberwithin{equation}{section}

\newcommand{\E}{\mathbf{E}\,}

\newcommand{\Tr}{\mathrm{Tr}\;\!}

\newcommand{\im}{\mathrm{Im}\;\!}

\newenvironment{Proof of}{\removelastskip\par\medskip
\noindent{\em Proof of} \rm}{\penalty-20\null\hfill$\square$\par\medbreak}

\begin{document}

\title{\bf {On the Asymptotic Distribution of Singular Values of  Products of
  Large Rectangular Random Matrices}}

\author{{\bf N. Alexeev}\\{\small S.-Peterburg state University}\\{\small S.-Petersburg, Russia}\and{\bf F. G\"otze}
\\{\small Faculty of Mathematics}
\\{\small University of Bielefeld}\\{\small Germany}
\and {\bf A. Tikhomirov}\\{\small Department of Mathematics
}\\{\small Komi Research Center of Ural Branch of RAS,}\\{\small Syktyvkar state University}
\\{\small Syktyvkar, Russia}}
\date{}
\maketitle
 \footnote{Partially supported by RF grant of the leading scientific schools
NSh-4472.2010.1. Partially supported by RFBR, grant  N 09-01-12180 and RFBR--DFG, grant N 09-01-91331. 
Grant of RF Government  ``Chebyshev Laboratory.''
Partially supported by CRC 701 ``Spectral Structures and Topological
Methods in Mathematics'', Bielefeld}
 

\today

\begin{abstract}We consider products of independent large random rectangular  matrices with independent entries. 
The limit distribution of the expected empirical distribution of singular values of such products is computed. 
The  distribution function   is described by its Stieltjes transform, which satisfies  some algebraic equation. 
In the particular case of square matrices we get a well-known  distribution which moments are Fuss-Catalan numbers.
\end{abstract}
\maketitle
\markboth{ N. Alexeev, F. G\"otze, A.Tikhomirov}{Product of random matrices}

\section{Introduction}
Let $m\ge 1$ be a fixed integer. For every $n\ge1$ consider  a  nondecreasing set of $m+1$
integers $p_0=n, p_1,\cdots, p_m$ where $p_{\nu}=p_{\nu}(n)$ for
$\nu=1,\ldots, m$,  depending  on $n$ and $p_{\nu}\ge n$.
For every $n\ge 1$ we consider an array of independent complex random variables
 $X^{(\nu)}_{jk},\ {  }1\le j\le p_{\nu-1},1\le k\le p_{\nu}$, $\nu=1,\ldots,m$  defined on a common probability space $\{\Omega_n,\mathbb F_n,\Pr\}$
with $\E X^{(\nu)}_{jk}=0$ and let $\E {|X^{(\nu)}_{jk}|}^2=1$. Let
$\mathbf X^{(\nu)}$ denote the  $p_{\nu-1}\times p_{\nu}$  matrix with entries
$[\mathbf X^{(\nu)}]_{jk}=\frac1{\sqrt{p_{\nu}}}X^{(\nu)}_{jk}$, for $1\le j\le p_{\nu-1}, 1\le k\le p_{\nu}$.
 The random variables $X^{(\nu)}_{jk}$ may depend on $n$ but
for simplicity we shall not make this explicit in our  notations.
Denote by $s_1\ge\ldots\ge s_n$ the singular values  of
the random  matrix
$\mathbf W:= \prod_{\nu=1}^m\mathbf X^{(\nu)}$
and define  the empirical distribution of  its squared singular values by
$$
\mathcal F_n(x)=\frac1n\sum_{k=1}^nI_{\{{s_k}^2\le x\}},
$$
where $I_{\{B\}}$ denotes the indicator of an event $B$.
We shall investigate the
approximation of the expected spectral distribution $F_n(x)=\E \mathcal F_n(x)$
 by
  the distribution function  $G_{\mathbf y}(x)$ which 
defined by its Stieltjes transform $s_{\mathbf y}(z)$
in the equation (\ref{main0}) below.

We consider  the Kolmogorov distance between the distributions
$ F_n(x)$ and $G_{\mathbf y}(x)$
$$
\Delta_n:=\sup_x|F_n(x)-G_{\mathbf y}(x)|.
$$
The main result of this paper is the following
\begin{thm}\label{main}Let $\E X^{(\nu)}_{jk}=0$, $\E|X^{(\nu)}_{jk}|^2=1$.
Assume the Lindeberg condition holds, i.e. for any $\tau>0$
\begin{equation}\notag
 L_n(\tau):=\max_{\nu=1,\ldots, m}\frac1{n^2}\sum_{j=1}^{p_{\nu-1}}\sum_{k=1}^{p_{\nu}}\E|X^{(\nu)}_{jk}|^2
I_{\{|X^{(\nu)}_{jk}|\ge\tau\sqrt n\}}\to0\quad\text{as  }n\to \infty
\end{equation}
Assume that $\lim_{n\to\infty}\frac n{p_l}=y_l\in(0,1]$.
Then, 
\begin{equation}\notag
{\lim_{n \to \infty} \sup_x|F_n(x)-G_{\mathbf y}(x)|=0.}
\end{equation}
\end{thm}
\begin{rem}For $m=1$ we get the well-known result of Marchenko-Pastur for sample covariance matrices
 \cite{M-P:67}.
\end{rem}
\begin{rem}\label{r-transform}In the case $y_1=y_2=\cdots=y_m=1$ the distribution $G_{\mathbf y}$ has  moments $M_k$ defined by
\begin{equation}\notag
 M_k=\int_0^{\infty}x^kdG_{\mathbf y}(x)=\frac1{mk+1} {\binom k{mk+k}},
\end{equation}
the so called Fuss--Catalan numbers.
\end{rem}
The Fuss-Catalan  numbers satisfy the following simple recurrence relation
\begin{equation}\label{recur}
M_k=\sum_{k_0+\cdots+k_m=k-1}\,\prod_{\nu=0}^mM_{k_{\nu}}.
\end{equation}
Denote by $s(z)$ Stieltjes transform of the  distribution $G$ determined by
its   moments $M_k$,
\begin{equation}\notag
s(z)=\int_{-\infty}^{\infty}\frac1{x-z}d G(x).
\end{equation}
 Using equality (\ref{recur}), we may show that this Stieltjes transform
$s(z)$ satisfies the equation
\begin{equation}\notag
1+zs(z)+(-1)^{m+1}z^ms(z)^{m+1}=0.
\end{equation}
In the general case ($y_l\ne 1$) Stieltjes transform satisfies the following equation
\begin{equation}\label{main0}
 1+zs_{\mathbf y}(z)-s(z)\prod_{l=1}^m(1-y_l-zy_ls_{\mathbf y}(z))=0,
\end{equation}
where $0\le y_l\le 1$.
For more details about the moments of such distributions see  \cite{AGT:09a}.

The result of Theorem \ref{main} is the first attempt in the Random Matrix Theory to describe the asymptotic of distribution of 
the singular spectrum  of a product of rectangular random matrices.
For rectangular random matrices there is no easily available analog   in free probability to describe the limit law. 
The Theorem \ref{main} was formulated in 
\cite{AGT:2010a}.
In the case of squared matrices ($y_1=y_2=\cdots=y_m=1$) 
there is an analog in the form of product of so-called free $\mathcal
R$-diagonal elements. It was studied for instance in Oravecz,
\cite{Oravecz:01}. 
It is well-known that the  moments of distribution of a product of free $\mathcal R$-diagonal elements are  Fuss-Catalans numbers (compare  Remark
\ref{r-transform}). In  \cite{AGT:10} it has been shown by the  method of
moments that the limit distribution of singular values of powers of random
matrices is the  distribution $G_{\mathbf y}$ with 
$y_1=\cdots=y_m=1$. In Banica and others \cite{banica:08} the result of Theorem \ref{main}
was obtained for square Gaussian matrices (see Theorem 6.1 in
\cite{banica:08}), using tools of  Free Probability theory. For a description
of the distribution  of $G_{\mathbf y}$ for the special case 
 $y_1=\cdots=y_m=1$, see  Speicher and Mingo \cite{speicher:08} as well.

In the the following  we shall give  the proof of Theorem \ref{main}.
We shall investigate the Stieltjes transform $s_n(z)$ of distribution function $F_n(x)$. We  show that $s_n(z)$ satisfies an approximate equation
\begin{equation}\notag
1+zs_n(z)-s_n(z)\prod_{l=1}^m(1-y_l-zy_ls_n(z))=\delta_n(z)
\end{equation}
where $\delta_n(z)\to 0$ as $n\to\infty$.
This relation together with relation (\ref{main0}) implies  that
$s_n(z)$ converges to $s(z)$ uniformly on any compact set in the  upper
half-plane $\mathcal K\subset \mathcal C^+$. The last claim is equivalent to
weak convergence of the distribution functions 
$F_n(x)$ to the distribution function $G_{\mathbf y}(x)$.

%

By
$C$ (with an index or without it) we shall denote generic
absolute constants,
whereas $C(\,\cdot\,,\,\cdot\,)$ will denote
positive constants depending on arguments.

\section{Auxiliary results}In this Section we describe a
 symmetrization of  a one-sided distribution and give  a special
representation for symmetrized  distribution of the  squared singular
values of random matrices. Furthermore, we  prove some lemmas about
 truncation of entries of random matrices.
\subsection{Symmetrization}We shall use the following  ``symmetrization'' of one-sided distributions. Let $\xi^2$
 be a positive random variable with distribution function $F(x)$. Define
$\widetilde \xi:=\varepsilon\xi$ where $\varepsilon$ denotes a Rademacher random variable with
$\Pr\{\varepsilon=\pm1\}=1/2$ which is
independent of $\xi$.
 Let $\widetilde F(x)$ denote
the distribution function of $\widetilde \xi$. It satisfies the
equation
\begin{equation}\label{sym}
\widetilde F(x)=1/2(1+\text{sgn} \{x\}\,F(x^2)),
\end{equation}
We apply this symmetrization to the distribution of the  squared
singular values of the matrix $\mathbf W$. Introduce the following matrices
\begin{align}\notag
\mathbf V=\left(\begin{matrix}{\mathbf W\quad \mathbf O}\\{\mathbf O\quad
\mathbf W^*}\end{matrix}\right),\quad
\mathbf J=\left(\begin{matrix}{\mathbf O\quad \mathbf I_{p_m}}\\{\mathbf I_{p_0}\quad
\mathbf O}\end{matrix}\right),
\quad\text{and}\quad
\widehat {\mathbf V}=\mathbf V\mathbf J\notag
\end{align}
Here and in the what follows  $\mathbf A^*$ denotes the adjoined (transposed and complex conjugate)  matrix $\mathbf A$ and
 $\mathbf I_k$ denotes unit matrix of order $k$. 
Note that $\widehat{\mathbf V}$ is Hermitian matrix. The eigenvalues of the
 matrix $\widehat{\mathbf V}$ are $-s_1,\ldots,-s_n,s_n,\ldots,s_1$ and $p_m-n$ zeros.
Note that the symmetrization of the distribution function $\mathcal
F_n(x)$ is a function $\widetilde{\mathcal F}_n(x)$ which is the empirical
distribution function  of the  non-zero eigenvalues of  matrix $\widehat {\mathbf V}$. By (\ref{sym}), we have
\begin{equation}\notag
\Delta_n=\sup_x|\widetilde F_n(x)-\widetilde
G_{\mathbf y}(x)|,
\end{equation}
where $\widetilde F_n(x)=\E\widetilde{\mathcal F}_n(x)$ and
$\widetilde G_{\mathbf y}(x)$ denotes  the symmetrization of the  distribution function
$G_{\mathbf y}(x)$.
\subsection{Truncation}We 
{shall now modify the random matrix $\mathbf X$ by truncation of its entries}.
Since the function $G_{\mathbf y}(x)$ is continuous with respect to $y_l$  we may assume
that $y_l=\frac{n}{p_l}$, $l=1,\ldots,m$. Furthermore, there exists a
constants $c>0$ and  $C>0$ such that $Cn\ge p_l\ge cn$ for any
$l=1,\ldots, m$. We note that there exists a sequence
$\tau_n\to0$ as $n\to\infty$ such that
$\frac1{\tau_n^2}L_n(\tau_n)\to 0$ as $n\to\infty$.  Introduce the
random variables
$X^{(\nu,c)}_{jk}=X^{(\nu)}_{jk}I_{\{|X^{(\nu)}_{jk}|\le
c\tau_n\sqrt n\}}$ and the matrix $\mathbf X^{(\nu,c)}=\frac1{\sqrt
{p_{\nu}}}({X^{(\nu,c)}_{jk}})$. Denote by
$s_1^{(c)}\ge\ldots\ge s_n^{(c)}$ the singular values  of the
random  matrix $ \mathbf W^{(c)}:= \prod_{\nu=1}^m{{\mathbf
X^{(\nu,c)}}}$. {Introduce the matrix} $\mathbf V^{(c)}:=\left(\begin{matrix}{\mathbf
W^{(c)}\quad \mathbf O}\\{\mathbf O\quad {\mathbf W^{(c)}}^*}\end{matrix}\right)$. We
define its empirical distribution by $\widetilde{\mathcal
F}_n^{(c)}(x)=\frac1{2n}\sum_{k=1}^nI_{\{{s_k^{(c)}}\le
x\}}+\frac1{2n}\sum_{k=1}^nI_{\{{-s_k^{(c)}}\le x\}}$. Let
$s_n(z)$ and $s_n^{(c)}(z)$ denote  Stieltjes transforms of
the distribution functions $\widetilde F_n(x)$ and $\widetilde
F_n^{(c)}(x)=\E\widetilde{\mathcal F}_n^{(c)}(x)$
respectively.
Define the resolvent matrices $\mathbf R=(\widehat{\mathbf V}-z\mathbf I)^{-1}$ and
$\mathbf R^{(c)}=({\widehat{\mathbf V}}^{(c)}-z\mathbf I)^{-1}$, where $\mathbf I$ denotes the unit matrix of corresponding dimension. Note that
\begin{equation}\notag
 s_n(z)=\frac1{2n}\E\Tr \mathbf R+\frac{1-y_m}{2y_mz},\qquad\text{and}\qquad
s_n^{(c)}(z)=\frac1{2n}\E\Tr \mathbf R^{(c)}+\frac{1-y_m}{2y_mz}.
\end{equation}
Applying the resolvent equality
\begin{equation}\notag
 (\mathbf A+\mathbf B-z\mathbf I)^{-1}=(\mathbf A-z\mathbf I)^{-1}-(\mathbf A-z\mathbf I)^{-1}
\mathbf B(\mathbf A+\mathbf B-z\mathbf I)^{-1},
\end{equation}
 we get
\begin{equation}\label{resolv}
 |s_n(z)-s_n^{(c)}(z)|\le \frac1{2n}\E|\Tr \mathbf R^{(c)}(\mathbf V-\mathbf V^{(c)})\mathbf J\mathbf
 R|.
\end{equation} 
Let
\begin{equation}\notag
\mathbf H^{(\nu)}=\left(\begin{matrix}{\mathbf X^{(\nu)}\quad\quad\quad\mathbf O}
\\{\mathbf O\quad\quad{\mathbf X^{(m-\nu+1)}}^*}\end{matrix}\right)
\quad\text{and}\quad\mathbf H^{(\nu,c)}=\left(\begin{matrix}{\mathbf X^{(\nu,c)}
\quad\quad\mathbf O}\\{\mathbf O\quad{\mathbf X^{(m-\nu+1,c)}}^*}\end{matrix}\right)
\end{equation}
Introduce the  matrices
$$
\mathbf V_{\alpha,\beta}=\prod_{q=a}^b\mathbf H^{(q)},\quad\mathbf V_{\alpha,\beta}^{(c)}=\prod_{q=a}^b\mathbf H^{(q,c)}.
$$
We have
\begin{equation}\label{repr1}
\mathbf V-\mathbf V^{(c)}=\sum_{q=1}^{m-1}\mathbf V^{(c)}_{1,q-1}(\mathbf H^{(q)}-\mathbf H^{(q,c)})\mathbf V_{q+1,m}.
\end{equation}
 Applying $\max\{\|\mathbf
R\|,\,\|\mathbf R^{(c)}\|\}\le v^{-1}$, inequality (\ref{resolv}),  and  the  representations (\ref{repr1})  together,  we get
\begin{equation}\label{st1}
|s_n(z)-s_n^{(c)}(z)|\le \frac{C}{\sqrt n}\sum_{q=1}^{m}\E ^{\frac12}\|(\mathbf X^{(q+1)}-\mathbf X^{(q+1,c)})\|_2^2\frac{1}{\sqrt n}\E^{\frac12}\|\mathbf V^{(c)}_{1,q-1}\mathbf R\mathbf R^{(c)}\mathbf V_{q+1,m}\|_2^2.
\end{equation}
Applying well-known   inequalities for matrix norms,  we get
\begin{equation}\notag
 \E\|\mathbf V^{(c)}_{1,q-1}\mathbf R\mathbf R^{(c)}\mathbf V_{q+1,m}\|_2^2
\le\frac C{v^4}\E\|\mathbf V^{(c)}_{1,q-1}\mathbf V_{q+1,m}\|_2^2
\end{equation}
In view  of Lemma \ref{norm2}, we obtain
\begin{equation}\label{st2}
 \E\|\mathbf V_{1,q-1}^{(c)}\mathbf R\mathbf R^{(c)}\mathbf V_{q+1,m}\|_2^2\le \frac{Cn}{v^4}.
\end{equation}
Direct calculations show that
\begin{equation}\label{st3}
 \frac1n\E\|\mathbf X^{(q)}-\mathbf X^{(q,c)}\|_2^2\le \frac C{n^2}\sum_{j,k=1}^n\E|X^{(q)}_{jk}|^2
I_{\{|X^{(q)}_{jk}|\ge c \tau_n\sqrt n\}}\le CL_n(\tau_n).
\end{equation}
Inequalities (\ref{st1}), (\ref{st2}) and ({\ref{st3}) together imply
\begin{equation}\label{stieltjes1}
 |s_n(z)-s_n^{(c)}(z)|\le\frac {C\sqrt {L_n(\tau_n)}}{v^2}.
\end{equation}

Furthermore, by definition of $X_{jk}^{(c)}$, we have
\begin{equation}\notag
|\E X_{jk}^{(q,c)}|\le \frac1{c\tau_n\sqrt n}\E |{X_{jk}^{(q)}}|^2 I_{\{|X_{jk}|\ge c\tau_n\sqrt n\}}.
\end{equation}
This implies that
\begin{equation}\label{st4}
 \|\E\mathbf X^{(q,c)}\|_2^2\le \frac C{n}\sum_{j=1}^{p_{q-1}}\sum_{k=1}^{p_q}|\E X_{jk}^{(q,c)}|^2
\le
\frac {CL_n(\tau_n)}{c\tau_n^2}.
\end{equation}
We denote  ${\widetilde{\mathbf H}}^{(\nu,c)}:=\left(\begin{matrix}{\mathbf X^{(\nu,c)}-\E\mathbf X^{(\nu,c)})\qquad\mathbf O}
\\{\mathbf O\qquad(\mathbf X^{(\nu,c)}-\E\mathbf
  X^{(\nu,c)})^*}\end{matrix}\right)$ 
and define the  respectively matrices 
$\widetilde{\mathbf W}^{(c)}$, $\widetilde{\mathbf V}^{(c)}$, $\widetilde {\mathbf V}^{(c)}_{a,b}$. 
Denote  by $\widetilde{\mathcal F}_n^{(c)}(x)$  the 
empirical distribution of  the squared singular values of the 
matrix $\widetilde {\mathbf V}^{(c)}\mathbf J$. 
Let ${\widetilde s}_n^{(c)}(z)$ denote the Stieltjes transform of 
the distribution function $\widetilde F_n^{(c)}=\E\widetilde {\mathcal F}_n^{(c)}$,
\begin{equation}\notag
 {\widetilde s}_n^{(c)}(z)=\int_{-\infty}^{\infty}\frac1{x-z}d\widetilde F_n^{(c)}(x).
\end{equation}

Similar to inequality (\ref{st1}) we get
\begin{equation}\notag
 |s_n^{(c)}-\widetilde s_n^{(c)}(z)|\le \sum_{q=0}^{m-1}\frac1{\sqrt n}\|\E\mathbf X^{(q,c)}\|_2\frac1{\sqrt n}\E^{\frac12}\|{\widetilde {\mathbf V}_{0,q}}^{(c)}\mathbf R^{(c)}
\widetilde {\mathbf R}^{(c)}\widetilde{\mathbf V}^{(c)}_{q+1,m}\|_2^2.
\end{equation}
Analogously to inequality (\ref{st2}), we get
\begin{equation}\notag
 \frac1{ n}\E\|\widetilde {\mathbf V}_{0,q}^{(c)}\mathbf R^{(c)}
\widetilde {\mathbf R}^{(c)}\widetilde {\mathbf V}^{(c)}_{q+1,m}\|_2^2\le \frac C{v^4}.
\end{equation}
By inequality (\ref{st4}),
\begin{equation}\notag
 \|\E X^{(q,c)}\|_2\le \frac{C\sqrt {L_n(\tau_n)}}{c\tau_n}.
\end{equation}
The last two inequalities together imply that
\begin{equation}\label{stieltjes2}
 |s_n^{(c)}-{\widetilde s}_n^{(c)}(z)|\le  \frac{C\sqrt {L_n(\tau_n)}}{\sqrt n\tau_n v^2}
\end{equation}
Inequalities (\ref{stieltjes1}) and (\ref{stieltjes2}) together imply that matrices $\mathbf W$ and
$\widetilde {\mathbf W}^{(c)}$ have the same limit distribution.
In the what follows we may assume without loss of generality that
  for any $n\ge 1$ and $\nu=1,\ldots, m$ and any $l=1,\ldots, m$ and $j=1,\ldots p_{l-1}$, $k=1,\ldots,p_l$,
\begin{equation}\label{conditions}
 \E X^{(\nu)}_{jk}=0,\quad \E {X^{(\nu)}_{jk}}^2=1, \quad\text{and}\quad |X^{(\nu)}_{jk}|\le c\tau_n\sqrt n
\end{equation}
with
\begin{equation}\notag
 \tau_n\to 0\qquad\text{and}\qquad \frac{L_n(\tau_n)}{\tau_n^2}\to 0\qquad{\text{as}}\qquad n\to \infty.
\end{equation}

\section{The proof of the main result for $m=2$}

Recal that the matrices ${\mathbf H}^{(q)}$, $q=1,\ldots,m$, and $\mathbf J$ are defined  by equalities
\begin{equation}\notag
{\mathbf H}^{(q)}=\left(\begin{matrix}{{\mathbf X}^{(q)}\quad \quad \quad \mathbf O}\\{\mathbf O\quad \mathbf {X^{(m-q+1)}}^*}\end{matrix}\right),
\qquad
\mathbf J:=\left(\begin{matrix}{\mathbf O\quad \mathbf I_{p_m}}\\{\mathbf I_{p_0}\quad \mathbf O}\end{matrix}\right),
\end{equation}
and that  $\mathbf A^*$ denotes 
the adjoint matrix $\mathbf A$ and $\mathbf I_k$ denotes the 
identity matrix of order $k$ (sometimes we shall omit
 the sub-index in the notation of the unit matrix).
Let $\mathbf V=\prod_{\nu=1}^m\mathbf H^{(\nu)}$,  $\widehat{\mathbf V}:=\prod_{q=1}^m\mathbf H^{(q)}\mathbf{J}$, and let
 $\mathbf R(z)$ denote the  resolvent matrix of the matrix $\widehat{\mathbf V}$,
\begin{equation}\notag
\mathbf R(z):=(\widehat{\mathbf V}-z\mathbf I_{p_m+p_0})^{-1}.
\notag
\end{equation}

 We note that the symmetrization of the distribution function 
$G_{\mathbf y}(x)$ has the Stieltjes transform $s_{\mathbf y}(z)$ 
(in the what follows we shall omit index $\mathbf y$ in the notation for
 this Stieltjes transform) which satisfies the following equation
\begin{equation}\label{main21}
1+zs(z)-\frac{s(z)}z\prod_{l=1}^m(1-y_l-zy_ls(z))=0.
\end{equation}

First, we prove Theorem \ref{main} for $m=2$.
We start from the simple equality
\begin{equation}\label{main4}
1+zs_n(z)=\frac1{2n}\E\Tr \mathbf V\mathbf J\mathbf R(z).
\end{equation}
Using the definition of the matrices $\mathbf V$, $\mathbf H^{(q)}$ and $\mathbf J$, we get
\begin{equation}\label{auxiliary0}
1+zs_n(z)=\frac1{2n\sqrt {p_1}}\sum_{j=1}^n\sum_{k=1}^{p_1}\E X^{(1)}_{jk}[
\mathbf H^{(2)}\mathbf J\mathbf R]_{kj}
+\frac1{2n\sqrt {p_2}}\sum_{j=1}^{p_1}\sum_{k=1}^{p_2}\E X^{(2)}_{jk}[\mathbf H^{(2)}\mathbf J\mathbf R]_{j+p_1,k+n}.
\end{equation}
In the what follows we shall use the notation $\varepsilon_n(z)$ as 
 generic  error function such that\newline $|\varepsilon_n(z)|\le \frac{C\tau_n}{v^4}$.
By Lemma \ref{teilor} in the Appendix, we get
\begin{align}\notag
1+zs_n(z)&=\frac1{2n {p_1}}\sum_{j=1}^n\sum_{k=1}^{p_1}\E\left[\frac{\partial \mathbf H^{(2)}\mathbf J\mathbf R}{\partial X^{(1)}_{jk}}\right]_{kj}\notag\\&+
\frac1{2n {p_2}}\sum_{j=1}^{p_1}\sum_{k=1}^{p_2}\E\left[\frac{\partial \mathbf H^{(2)}\mathbf J\mathbf R}{\partial X_{jk}^{(2)}}\right]_{j+p_1,k+n}+\varepsilon_n(z),
\end{align}
where $|\varepsilon_n(z)|\le \frac{C\tau_n}{v^4}$.


Put $n_1=\max\{2p_1,n+p_2\}.$ Let $\mathbf e_1,\ldots,\mathbf e_{n_1}$
 be an orthonormal basis of $\mathbb
R^{n_1}$. First we note that, for $j=1,\ldots, n$ and for $k=1,\ldots,p_1$
\begin{equation}\label{auxiliary2}\frac{\partial  \mathbf H^{(1)}}{\partial X^{(1)}_{jk}}=\frac1{\sqrt {p_1}}\mathbf e_{j}\mathbf e_{k}^T,\qquad
\frac{\partial  \mathbf H^{(2)}}{\partial X^{(1)}_{jk}}=\frac1{\sqrt {p_1}}\mathbf e_{k+{p_1}}\mathbf e_{j+p_2}^T,
\end{equation}
and for $j=1,\ldots,p_1$ and $k=1,\ldots,p_2$,
\begin{equation}\label{auxiliary2*}\frac{\partial  \mathbf H^{(1)}}{\partial X^{(2)}_{jk}}=\frac1{\sqrt {p_2}}\mathbf e_{k+n}\mathbf e_{j+p_1}^T,
\qquad\frac{\partial  \mathbf H^{(2)}}{\partial X^{(2)}_{jk}}=\frac1{\sqrt {p_2}}\mathbf e_j\mathbf e_k^T.
\end{equation}

We first compute the derivatives of the resolvent matrix as follows
\begin{align}\label{auxiliary4}
\frac{\partial \mathbf R}{\partial X^{(1)}_{jk}}=&-\frac1{\sqrt {p_1} }\mathbf R\mathbf e_j\mathbf e_k^T\mathbf H^{(2)}\mathbf J\mathbf R-\frac1{\sqrt {p_1}} \mathbf R\mathbf H^{(1)}\mathbf e_{k+p_1}\mathbf e_{j+p_2}^T\mathbf J\mathbf R,
\notag\\
\frac{\partial \mathbf R}{\partial X^{(2)}_{jk}}=&-\frac1{\sqrt {p_2} }\mathbf R\mathbf e_{k+n}\mathbf e_{j+p_1}^T\mathbf H^{(2)}\mathbf J\mathbf R-\frac1{\sqrt {p_2}} \mathbf R\mathbf H^{(1)}\mathbf e_{j}\mathbf e_{k}^T\mathbf J\mathbf R,
\end{align}
and
\begin{align}\label{auxiliary5}
\frac{\partial (\mathbf H^{(2)}\mathbf J\mathbf R)}{\partial X^{(1)}_{jk}}=\frac1{\sqrt
{p_1}} \mathbf e_{k+p_1}\mathbf e_{j+p_2}^T\mathbf J\mathbf R
&-\frac1{\sqrt {p_1} }\mathbf H^{(2)}\mathbf J\mathbf R\mathbf e_j\mathbf e_k^T\mathbf H^{(2)}\mathbf J\mathbf R\notag\\
&-\frac1{\sqrt {p_1}} \mathbf H^{(2)}\mathbf J\mathbf R\mathbf H^{(1)}\mathbf e_{k+p_1}\mathbf e_{j+p_2}^T\mathbf J
\mathbf R,
\end{align}

and
\begin{align}\label{auxiliary9}
\frac{\partial (\mathbf H^{(2)}\mathbf J\mathbf R)}{\partial X^{(2)}_{jk}}=\frac1{\sqrt
{p_2}} \mathbf e_{j}\mathbf e_{k}^T\mathbf J\mathbf R
&-\frac1{\sqrt {p_2} }\mathbf H^{(2)}\mathbf J\mathbf R\mathbf e_{k+n}\mathbf e_{j+p_1}^T\mathbf H^{(2)}\mathbf J\mathbf R\notag\\
&-\frac1{\sqrt {p_2}} \mathbf H^{(2)}\mathbf J\mathbf R\mathbf H^{(1)}\mathbf e_{j}\mathbf e_{k}^T\mathbf J
\mathbf R.
\end{align}


The equalities (\ref{auxiliary0}) and (\ref{auxiliary5}) and
(\ref{auxiliary9}) together imply that
\begin{equation}\label{auxiliary6}
1+zs_n(z)=A_1+A_2+A_3+\varepsilon_n(z),
\end{equation}
where
\begin{align}
A_1&:=-\frac1{2np_1}\sum_{j=1}^n\sum_{k=1}^{p_1}\E [\mathbf H^{(2)}\mathbf J\mathbf R]_{jk}^2
-\frac1{2np_2}\sum_{j=1}^{p_1}\sum_{k=1}^{p_2}[\mathbf H^{(2)}\mathbf J\mathbf R]_{j+p_1,k+n}^2,\notag\\
A_2&:=-\frac1{2n {p_1}}\sum_{k=1}^{p_1}\E[\mathbf H^{(2)}\mathbf J\mathbf R\mathbf H^{(1)}]_{k,k+p_1}
\sum_{j=1}^n[\mathbf J\mathbf R]_{j+p_2,j},\notag\\
A_3&:=-\frac1{2np_2}\E\sum_{k=1}^{p_1} [\mathbf H^{(2)}\mathbf J\mathbf R\mathbf H^{(1)}]_{k+p_1,k}
\sum_{j=1}^{p_2}[\mathbf J\mathbf R]_{j,j+n}.\notag
\end{align}

We prove that the first  summand is negligible and the main
asymptotic terms  are given by  $A_2$ and $A_3$. We mow start  the investigation
of these summands.


\begin{lem}\label{lem1}Under conditions of Theorem \ref{main} we have
\begin{align}\label{l100}
\left|A_2+\left(\frac{1}{2}\frac1{p_1}\E\sum_{k=1}^{p_1} [\mathbf H^{(2)} \mathbf J\mathbf R\mathbf H^{(1)}]_{k,k+p_1}\right)\left(\frac1n\sum_{j=1}^n\E [\mathbf J\mathbf R]_{j+p_2,j}\right)\right|&\le \frac C{nv^4},\notag\\
\left|A_3+\left(\frac{1}{2}\frac1{p_2}\E\sum_{k=1}^{p_1} [\mathbf H^{(2)}\mathbf J
\mathbf R\mathbf H^{(1)}]_{k+p_1,k}\right)\left(\frac1{n}\sum_{j=1}^{p_2}\E [\mathbf J\mathbf R]_{j,j+n}\right)\right|&\le \frac C{nv^4}.
\end{align}
\end{lem}
\begin{proof}Applying Lemma \ref{var1} with $m=2$ and $a=1$ and Lemma \ref{var0} (see Appendix), we obtain
\begin{align}
\left|A_2+\left(\frac{1}{2}\frac1{p_1}\E\sum_{k=1}^{p_1}[\mathbf H^{(2)}\mathbf J \mathbf R\mathbf H^{(1)}]_{k,k+p_1}\right)\left(\frac1n\sum_{j=1}^n\E [\mathbf J\mathbf R[_{j+p_2,j}\right)\right|\qquad\qquad\qquad\qquad&\notag\\ \le \E^{\frac12}\left| \frac{1}{2}\frac1{p_1}\left(\sum_{k=1}^{p_1}[\mathbf H^{(2)}\mathbf J \mathbf R\mathbf H^{(1)}]_{k,k+p_1}-
\E\sum_{k=1}^{p_1} [\mathbf H^{(2)}\mathbf J \mathbf R\mathbf H^{(1)}]_{k,k+p_1}\right)\right|^2&\notag\\ \qquad\qquad\qquad\qquad\qquad\times\E^{\frac12}\left|\frac1n\sum_{j=1}^n([\mathbf J\mathbf R]_{j+p_2,j}-\E \mathbf J\mathbf R]_{j+p_2,j})\right|^2\le \frac C{nv^4}&\notag
\end{align}
Similar we prove the second inequality in (\ref{l100}).
Thus the Lemma is proved.
\end{proof}
Note that
\begin{align}
\frac1n\sum_{j=1}^n\E[\mathbf J\mathbf R]_{j+p_2,j}&=\frac1n\sum_{j=1}^n\E \mathbf R_{jj}=s_n(z),\notag\\
\frac1n\sum_{k=1}^{p_2}\E[\mathbf J\mathbf R]_{k,k+n}&=\frac1n\sum_{j=1}^{p_2}\E [\mathbf R]_{j+n,j+n}=s_n(z)-\frac{1-y_2}{y_2z}.\label{auxiliary7}
\end{align}
Lemma \ref{lem1}, equalities (\ref{auxiliary7}) and the definition of matrices ${\mathbf H^{(\nu)}}$, for $\nu=1,2$, together imply
\begin{equation}
A_2=-\frac12s_n(z)\frac1{2p_1\sqrt{p_2}}\sum_{j=1}^{p_1}\sum_{k=1}^{p_2}\E X^{(2)}_{jk}\, [\mathbf H^{(2)}\mathbf J\mathbf R]_{j,k+n}+\varepsilon_n(z),
\end{equation}
and similar
\begin{equation}\notag
A_3=-\frac12(s_n(z)-\frac{1-y_2}{y_2z})\frac1{2p_2\sqrt{p_1}}\sum_{j=1}^n\sum_{k=1}^{p_1}\E X^{(1)}_{jk}\,  [\mathbf H^{(2)}\mathbf J\mathbf R]_{k+p_1,j}+\varepsilon_n(z).
\end{equation}

Applying Lemma \ref{teilor} and equalities (\ref{auxiliary2})--(\ref{auxiliary9}), we get
\begin{align}\label{auxialiry8}
A_2&= -s_n(z)\frac1{2p_1{p_2}}(p_1-\sum_{j=1}^{p_1}\E [\mathbf H^{(2)}\mathbf J\mathbf R\mathbf H^{(1)}]_{j,j})\sum_{k=1}^{p_2}\E[\mathbf J\mathbf R]_{k,k+n}+A_4+\varepsilon_n(z)\notag\\
A_3&= -(s_n(z)-\frac{1-y_2}{y_2z})\frac1{2p_2{p_1}}(p_1-\sum_{k=1}^{p_1}\E [\mathbf H^{(2)}\mathbf J\mathbf R\mathbf H^{(1)}]_{k+p_1,k+p_1})\sum_{j=1}^n\E[\mathbf J\mathbf R]_{j+p_2,j}\\
&\qquad\qquad \qquad\qquad\qquad\qquad\qquad\qquad\qquad\qquad\qquad\qquad\qquad+ A_5+\varepsilon_n(z),
\end{align}
where
\begin{align}
A_4&=s_n(z)\frac1{2p_1p_2}\sum_{j=1}^{p_1}\sum_{k=1}^{p_2}\E[\mathbf H^{(2)}\mathbf J\mathbf R]_{j,k+n}
[\mathbf H^{(2)}\mathbf J\mathbf R]_{j+p_1,k+n},\notag\\
A_5&=(s_n(z)-\frac{1-y_2}{y_2z})\frac1{2p_1p_2}\sum_{j=1}^{n}\sum_{k=1}^{p_1}\E[\mathbf H^{(2)}\mathbf J\mathbf R\mathbf H^{(1)}]_{k+p_1,j}[\mathbf J\mathbf R]_{k,j}\notag
\end{align}
Note that
\begin{align}
\sum_{j=1}^{p_1}\E [\mathbf H^{(2)}\mathbf J\mathbf R\mathbf H^{(1)}]_{j,j}+\sum_{k=1}^{p_1}\E [\mathbf H^{(2)}\mathbf J\mathbf R\mathbf H^{(1)}]_{k+p_1,k+p_1}&=\E\Tr\mathbf H^{(2)}\mathbf J\mathbf R\mathbf H^{(1)}\notag\\&=\E\Tr\mathbf H^{(1)}\mathbf H^{(2)}\mathbf J\mathbf R=\E\Tr\mathbf V\mathbf J\mathbf R.\notag
\end{align}
By resolvent equality $\mathbf I+z\mathbf R=\mathbf V\mathbf J\mathbf R$, we have
\begin{equation}\label{in101}
\frac1{2n}(\sum_{j=1}^n\E [\mathbf H^{(1)}\mathbf H^{(2)}\mathbf J\mathbf R]_{j,j}+\sum_{j=1}^{p_2}\E [\mathbf H^{(1)}\mathbf H^{(2)}\mathbf J\mathbf R]_{j+n,j+n})=1+zs_n(z).
\end{equation}
Equalities (\ref{main4}), (\ref{auxialiry8}) and (\ref{in101}) together imply
\begin{equation}\label{in10}
A_2+A_3=\frac{s_n(z)}{z}(1-y_1-zy_1s_n(z))(1-y_2-zy_2s_n(z))+A_4+A_{5}+\varepsilon_n(z).
\end{equation}


\begin{lem}\label{lem6}Under condition of Theorem \ref{main} we have
\begin{equation}\label{in12}
\max\{|A_1|,\,|A_{4}|,\,|A_5|\}\le \frac C{nv^2}.
\end{equation}
\end{lem}
\begin{proof}{We shall consider the  bound for the quantity $A_{5}$ only}. The others are similar.
By H\"older's inequality, we have
\begin{equation}\notag
|A_{4}|\le \frac1{n^2v}\E\|\mathbf H^{(2)}\mathbf J\mathbf R\|_2^2,
\end{equation}
where $\|\cdot\|_2$ denotes Hilbert-Schmidt norm of matrix. Continuing the last inequality, we may write
\begin{equation}\notag
|A_4|\le \frac C{n^2v^3}\E\|\mathbf H^{(2)}\|_2^2.\end{equation}
A simple calculation shows that
\begin{equation}
 \E\|\mathbf H^{(2)}\|_2^2\le Cn
\end{equation}
The last two inequalities together imply
\begin{equation}\notag
|A_5|\le \frac C{nv^3}.\end{equation}
Thus the Lemma is proved.
\end{proof}
Relation (\ref{in10}) and Lemma \ref{lem6} together imply
\begin{equation}\label{direct}
1+zs_n(z)=\frac{s_n(z)}{z}(1-y_1-zy_1s_n(z))(1-y_2-zy_2s_n(z))+\delta_n(z)
\end{equation}
where $|\delta_n(z)|\le \frac C{nv^4}+\frac {C\tau_n}{v^2}$.

\begin{lem}\label{stieltjes}Under conditions of Theorem \ref{main} for $v\ge
  3$ we have   for sufficiently large $n$, 
\begin{equation}
|s(z)-s_n(z)|\le \frac{C|\delta_n(z)|}{v}.
\end{equation}

\end{lem}

\begin{proof}We rewrite the equation (\ref{direct}) as follows
\begin{equation}\label{in100}
 1+zs_n(z)=\frac1zs_n(z)(1-y_1)(1-y_2)-zs_n^3(z)+s^2(z)(y_1(1-y_2)+y_2(1-y_1))+\delta_n(z).
\end{equation}

Introduce the notations
\begin{align}
 d&=\frac{(1-y_1)(1-y_2)}{z}\notag\\
d_n&=z(s_n(z)^2+s_n(z)s(z)+s^2(z))\notag\\
h_n&={(s(z)+s_n(z))(y_1(1-y_2)+y_2(1-y_1))}.\notag
\end{align}
Then we may rewrite equality (\ref{in100}) as follows
\begin{equation}\notag
 s_n(z)-s(z)=\frac{\delta_n(z)}{-z+d+d_n+h_n}
\end{equation}
First we note that
\begin{equation}\label{in30}
\im \{d\}\le 0.\qquad
\end{equation}
Furthermore, note that
\begin{equation}\label{r1}
 |zs_n(z)|\le 1+\frac{\E^{\frac12}\|V\|_2^2}{nv}\le 1+\frac1v.
\end{equation}
Using that $\max\{|s(z)|,|s_n(z)|\}\le \frac1v$ and (\ref{r1}),  we get
\begin{equation}\label{in50}
 \max\{|h_n|,\,|d_n(z)|\}\le (1+\frac1v)\frac1v
\end{equation}
We take $v\ge 3$.
Equalities (\ref{in30}), (\ref{in50})  together complete the proof of lemma.
\end{proof}\qquad
The last Lemma implies that in $\mathcal C^+$ there exists an open set with
non-empty interior such that on this set $s_n(z)$ converges  to $s(z)$.
 The Stieltjes transform of our 
random variables is an analytic function on $\mathcal C^+$ 
and locally bounded (that is $|s_n(z)|\le v^{-1}$ for any $v>0$). 
By Montel's Theorem (see, for instance,
\cite{Conway}, p. 153, Theorem 2.9) $s_n(z)$ converges to $s(z)$ 
uniformly on  any compact set $\mathcal K\subset \mathcal C^+$ in the upper half-plane.
  This implies that $\Delta_n\to 0$ as $n\to\infty$. Thus the proof  of Theorem \ref{main} in the case $m=2$ is complete.

\section{  The proof of the main result in the general case} Recall that
$\mathbf H^{(q)}$ and $\mathbf J$ are defined by following  equalities, with $q=1,\ldots, m$,
\begin{equation}
\mathbf H^{(q)}=\left
(\begin{matrix}{\mathbf X^{(q)}\quad \mathbf O}\\{\mathbf O\quad
{{\mathbf X}^{(m-q+1)}}^*}\end{matrix}\right),
\qquad
\mathbf J=\left(\begin{matrix}{\mathbf O\quad \mathbf I_{p_m}}\\{\mathbf I_{p_0}\quad \mathbf O}\end{matrix}\right),
\end{equation}
where $\mathbf I_k$ denotes the identity matrix of {dimension} $k$.
Note that  $\mathbf H^{(q)}$ {is a $(p_{q-1}+p_{m-q+1})\times(p_{q}+p_{m-q})$ --matrix}.
Let $\mathbf V=\prod_{q=1}^m\mathbf H^{(q)}$, $\widehat{\mathbf V}:=\mathbf V\mathbf{J}$, and denote by
$\mathbf R$ its resolvent matrix,
\begin{equation}\notag
\mathbf R:=(\widehat{\mathbf V}-z\mathbf I)^{-1}.
\notag
\end{equation}
We shall use the following  ``symmetrization'' of one-sided distributions. Let
$\xi^2$ be a positive random variable. Define
$\widetilde \xi:=\varepsilon\xi$ where $\varepsilon$ denotes a Rademacher
random variable with $\Pr\{\varepsilon=\pm1\}=1/2$ which independent of
$\xi$. We apply this symmetrization to the distribution of the  singular
values of the matrix $\mathbf X^2$. Note that the  symmetrized  distribution function
$\widetilde F_n(x)$ satisfies the equation
\begin{equation}\notag
\widetilde F_n(x)=1/2(1+\text{sgn} \{x\}\,F_n(x^2)),
\end{equation}
and this function is the empirical spectral distribution function of the random matrix $\mathbf W$.
Furthermore, we note that the symmetrization of the 
distribution function $G(x)$ has the Stieltjes transform $s(z)$ which satisfies the following equation
\begin{equation}\label{main21*}
1+zs(z)-\frac {s(z)}z\prod_{\nu=1}^m{(1-y_{\nu}-zy_{\nu}s(z))}=0.
\end{equation}
In the rest of paper we shall prove that Stieltjes transform of expected spectral distribution function $s_n(z)=\int_{-\infty}^{\infty}\frac1{x-z}
\text{d}\E \widetilde F_n(x)$ satisfies the equation
\begin{equation}\label{main31}
1+zs_n(z)-\frac {s_n(z)}z\prod_{\nu=1}^m(1-y_{\nu}-zy_{\nu}s_n(z))=\delta_n(z),
\end{equation}
where $\delta_n(z)$ denotes some function such that $\delta_n(z)\to 0$ as $n\to \infty$.

We start from the simple equality
\begin{equation}\label{main41}
1+zs_n(z)=\frac1{2n}\Tr \widehat{\mathbf V}\mathbf R.
\end{equation}
By definition of the matrices $\mathbf V$, $\mathbf H^{(q)}$ and $\mathbf {J}$, we get
\begin{align}\label{auxiliary01}
1+zs_n(z)=\frac1{2n\sqrt {p_1}}\sum_{j=1}^{p_0}\sum_{k=1}^{p_1}\E X^{(1)}_{jk}&[\mathbf V_{2,m}\mathbf J\mathbf R]_{kj}\notag\\&
+\frac1{2n\sqrt {p_m}}\sum_{j=1}^{p_{m-1}}\sum_{k=1}^{p_m}\E X^{(m)}_{jk}[\mathbf V_{2,m}\mathbf J\mathbf R]_{j+p_1,k+p_0},
\end{align}
where $\mathbf V_{\alpha,\beta}=\prod_{q=a}^b\mathbf H^{(q)}$.
To simpify the  calculations we assume that $X_{jk}^{(\nu)}$ are
i.i.d. Gaussian random variables, and we  shall use the following well-known
 equality for a Gaussian r.v. $\xi$
\begin{equation}\label{auxiliary11}
\E\xi f(\xi)=\E f'(\xi),
\end{equation}
for every differentiable function $f(x)$ such that both sides exist. 
By Lemma \ref{teilor}, we obtain that the error involved in this
Gaussian assumption is of order $O(\tau_n)$.
Recall the  notation $\varepsilon_n(z)$  for generic error functions such 
that $|\varepsilon_{n}(z)|\le
C\tau_nv^{-q}$, for some $q\ge 0$. Let $p_0=n$ and $n_1=\max_{0\le \nu\le m-1}\{p_{\nu}+p_{m-\nu}\}$.
Let $\mathbf e_1,\ldots,\mathbf e_{n_1}$ be an  orthonormal basis 
of $\mathbb R^{n_1}$.
First we note that, for $j=1,\ldots,p_{q-1}$ and $k=1,\ldots,p_q$,
\begin{equation}\label{auxiliary21}
\frac{\partial  \mathbf H^{(q)}}{\partial X^{(q)}_{jk}}=\frac1{\sqrt{p_q}}\mathbf e_j\mathbf e_k^T,\qquad
\frac{\partial  \mathbf H^{(m-q+1)}}{\partial X^{(q)}_{jk}}=\frac1{\sqrt{p_q}}\mathbf e_{k+p_{m-q}}\mathbf e_{j+p_{m-q+1}}^T,
\end{equation}
and, for $j=1,\ldots,p_{m-q}$ and $k=1,\ldots,p_{m-q+1}$
\begin{equation}\label{auxiliary21*}
\frac{\partial  \mathbf H^{(m-q+1)}}{\partial X^{(m-q+1)}_{jk}}=\frac1{\sqrt{p_{m-q+1}}}\mathbf e_j\mathbf e_k^T,\qquad
\frac{\partial  \mathbf H^{(q)}}{\partial X^{(m-q+1)}_{jk}}=\frac1{\sqrt{p_{m-q+1}}}\mathbf e_{k+p_{q-1}}\mathbf e_{j+p_q}^T.
\end{equation}
Now we may compute the derivatives of the matrix $\mathbf V_{2,m}\mathbf J\mathbf R$ as follows
\begin{align}\label{auxiliary41}
\frac{\partial (\mathbf V_{2,m}\mathbf J\mathbf R)}{\partial X^{(1)}_{jk}}=\frac1{\sqrt {p_1}}\mathbf V_{2,m-1}\mathbf e_{k+p_{m-1}}\mathbf e_{j+p_m}^T\mathbf J\mathbf R
&-\frac1{\sqrt {p_1} }\mathbf V_{2,m}\mathbf J\mathbf R\mathbf e_j\mathbf e_k^T\mathbf V_{2,m}\mathbf J\mathbf R
\notag\\
&-\frac1{\sqrt {p_1} }\mathbf V_{2,m}\mathbf J\mathbf R\mathbf V_{1,m-1}\mathbf e_{k+p_{m-1}}\mathbf e_{j+p_m}^T
\mathbf J\mathbf R.\qquad
\end{align}
and
\begin{align}\label{auxiliary42}
\frac{\partial (\mathbf V_{2,m}\mathbf J\mathbf R)}{\partial X^{(m)}_{jk}}=\frac1{\sqrt {p_{m}}}\mathbf V_{2,m-1}\mathbf e_{j}\mathbf e_{k}^T\mathbf J\mathbf R
&-\frac1{\sqrt {p_{m}}}\mathbf V_{2,m}\mathbf J\mathbf R\mathbf e_{k+n}\mathbf e_{j+p_1}^T\mathbf V_{2,m-1}\mathbf J\mathbf R
\notag\\
&-\frac1{\sqrt {p_{m}} }\mathbf V_{2,m}\mathbf J\mathbf R\mathbf V_{1,m-1}\mathbf e_{j}\mathbf e_{k}^T
\mathbf J\mathbf R.
\end{align}
The equalities (\ref{auxiliary01}) and (\ref{auxiliary41}) together imply
\begin{equation}\label{auxiliary61}
1+zs_n(z)=A_1+A_2+A_3+B_1+B_2+B_3+\varepsilon_n(z),
\end{equation}
where
\begin{align}
A_1&:=\frac1{2np_1}\E\sum_{k=1}^{p_1}[\mathbf V_{2,m-1}]_{k,k+p_{m-1}}\sum_{j=1}^{n}[\mathbf J\mathbf R]_{j+p_m,j},\notag\\
A_2&=-\frac1{2np_1}\E\sum_{j=1}^{n}\sum_{k=1}^{p_1}[\mathbf V_{2,m}\mathbf J\mathbf R]_{k,j}^2,\notag\\
A_3&:=-\frac1{2np_1}\E\sum_{k=1}^{p_1} [\mathbf V_{2,m}\mathbf J\mathbf R\mathbf V_{1,m-1}]_{k,k+p_{m-1}}
\sum_{j=1}^{n}[ \mathbf J\mathbf R]_{j+p_m,j}\notag
\end{align}
and
\begin{align}
B_1&=\frac1{2np_m}\E\sum_{j=1}^{p_{m-1}}[\mathbf V_{1,m-1}]_{j+p_1,j}\sum_{k=1}^{p_m}[\mathbf J\mathbf R]_{k,k+n},\notag\\
B_2&=-\frac1{2np_m}\E\sum_{j=1}^{p_{m-1}}\sum_{k=1}^{p_m}[\mathbf V_{2,m}\mathbf J\mathbf R]_{j+p_1,k+n}^2,\notag\\
B_3&:=-\frac1{2np_m}\E\sum_{j=1}^{p_{m-1}} [\mathbf V_{2,m}\mathbf J\mathbf R\mathbf V_{1,m-1}]_{j+p_1,j}
\sum_{k=1}^{p_m}[\mathbf J\mathbf R]_{k,k+n}
.\notag
\end{align}
\begin{lem}
Under the  conditions of Theorem \ref{main} there exists a constant $C>0$
such that the following inequality holds
\begin{equation}
\max\{|A_2|,\,|B_2|\}\le \frac C{nv^2}.
\end{equation}
\end{lem}
\begin{proof}
 Note that
\begin{align}
 |A_2|\le \frac1{n^2}\E\|\mathbf V_{2,m}\mathbf J\mathbf R\|_2^2\le \frac C{n^2v^2}\E\|\mathbf V_{2,m}\|_2^2
\end{align}
By Lemma \ref{norm2},
\begin{equation}
\E\|\mathbf V_{2,m}\|_2^2\le Cn
\end{equation}
The last two inequalities conclude the proof.
The bound for $|B_2|$ is similar.
Thus the Lemma is proved.
\end{proof}
\begin{lem}
Under conditions of Theorem \ref{main} there exists a constant $C>0$
such that the following inequality holds
\begin{equation}\notag
\max\{|A_1|,\,|B_1|\}\le \frac C{nv^2}.
\end{equation}
\end{lem}
\begin{proof}
 We consider the quantity $A_1$ only. The bound for $B_1$ is similar.
By Lemma \ref{var1}, we have
\begin{equation}\notag
 |A_1-\frac1{2p_1}\sum_{k=1}^{p_1}\E[\mathbf V_{1,m-1}]_{k,k+n}\frac1n\sum_{j=1}^{n}\E[\mathbf J\mathbf R]_{j+n,j}|\le\frac C{nv^2}.
\end{equation}
Direct calculation shows that
\begin{equation}\notag\qquad
 \E[\mathbf V_{1,m-1}]_{k,k+n}=0.
\end{equation}
Thus the Lemma is proved.
\end{proof}
\begin{lem}\label{lem11}
Under conditions of Theorem \ref{main} there exists a constant $C>0$
such that the following inequality holds
\begin{align}
&|A_3+\frac1{2p_1}\sum_{k=1}^{p_1}\E[\mathbf V_{2,m}\mathbf J\mathbf R\mathbf V_{1,m-1}]_{k,k+p_{m-1}}
\frac1{n}\sum_{j=1}^{n}\E[\mathbf J\mathbf R]_{j+p_m,j}|\le\frac C{nv^2},\notag\\
&|B_3+\frac1{2n}\sum_{k=1}^{p_{m-1}}\E[\mathbf V_{2,m}\mathbf J\mathbf R\mathbf V_{1,m-1}]_{k+p_{1},k}\frac1{p_m}
\sum_{j=1}^{p_m}\E[\mathbf J\mathbf R]_{j,j+n}|\le\frac C{nv^2}.\notag
\end{align}
\end{lem}
\begin{proof}

 Applying H\"older's inequality and Lemmas \ref{var1} and \ref{var0} together, we get
\begin{align}
 |A_3&+\frac1{2p_1}\sum_{k=1}^{p_1}\E[\mathbf V_{2,m}\mathbf J\mathbf R\mathbf V_{1,m-1}]_{k,k+p_{m-1}}\frac1n\sum_{j=1}^{n}\E[\mathbf J\mathbf R]_{j+p_m,j})|\notag\\&\le
\E^{\frac12}|\frac1{2n}(\sum_{k=1}^{p_1}[\mathbf V_{2,m}\mathbf J\mathbf R\mathbf V_{1,m-1}]_{k,k+p_{m-1}}-
 \E\sum_{k=1}^{p_1}[\mathbf V_{2,m}\mathbf J\mathbf R\mathbf V_{1,m-1}]_{k,k+p_{m-1}})|^2\notag\\&\qquad\qquad\qquad\qquad\qquad\times\E^{\frac12}|\frac1n(\sum_{j=1}^{n}[\mathbf J\mathbf R]_{j+p_m,j}-
\E\sum_{j=1}^{n}[\mathbf J\mathbf R]_{j+p_m,j})|^2\le\frac C{nv^2}.\notag
\end{align}\qquad
Thus the Lemma is proved.

\end{proof}

Introduce the following notations, for $\alpha,\beta=1,\ldots,m$,
\begin{equation}\notag
 f_{\alpha,\beta}=\frac1{p_{\alpha-1}}\sum_{k=1}^{p_{\alpha-1}} \E[\mathbf V_{\alpha,m}\mathbf J\mathbf R\mathbf V_{1,\beta}]_{k,k+p_{\beta}},\qquad
g_{\alpha,\beta}=\frac1{p_{\beta+1}}\sum_{k=1}^{p_{\beta}} \E[\mathbf V_{\alpha,m}\mathbf J\mathbf R\mathbf V_{1,\beta}]_{k+p_{\alpha-1},k},
\end{equation}
and
\begin{equation}\notag
 f_{m+1,0}=\frac1{p_m}\sum_{k=1}^{p_m} \E[\mathbf J\mathbf R]_{k,k+p_{0}},\qquad
g_{m+1,0}=\frac1{p_0}\sum_{k=1}^{n} \E[\mathbf J\mathbf R]_{k+p_m,k},
\end{equation}
It is straightforward to check that
\begin{align}\label{auxiliary71}
f_{m+1,0}&=\frac1{p_m}\sum_{j=1}^{p_m}\E [\mathbf R]_{k+n,k+n}=\frac1{z}(1-y_m-zy_ms_n(z)),
\notag\\g_{m+1,0}&=\frac1{p_0}\sum_{j=1}^{n}\E [\mathbf R]_{j,j}=s_n(z).
\end{align}
By Lemma \ref{lem11} and  equality (\ref{auxiliary71}), we may write
\begin{equation}\label{a1}
A_3+B_3=-\frac12s_n(z)f_{2,m-1}-\frac12(-y_ms_n(z)+\frac{1-y_m}{z})g_{2,m-1}+\varepsilon_n(z).
\end{equation}
Now we investigate  the behavior of the coefficients $f_{\alpha,m-\alpha+1}$ and  $g_{\alpha, m-\alpha+1}$, for $\alpha=2,\ldots,m$. Assume that $\alpha\le m-\alpha$.
We have
\begin{align}\label{a2}
f_{\alpha,m-\alpha+1}&=\frac1{p_{\alpha-1}\sqrt {p_{\alpha}}}\sum_{j=1}^{p_{\alpha-1}}
\sum_{k=1}^{p_{\alpha}}\E X^{(\alpha)}_{j,k}[\mathbf V_{\alpha+1,m}
\mathbf J\mathbf R\mathbf V_{1,m-\alpha+1}]_{k,j+p_{m-\alpha+1}}\notag\\
g_{\alpha,m-\alpha}&=\frac1{p_{m-\alpha}\sqrt {p_{m-\alpha+1}}}
\sum_{j=1}^{p_{m-\alpha}}\sum_{k=1}^{p_{m-\alpha+1}}\E X^{(m-\alpha+1)}_{j,k}
[\mathbf V_{\alpha+1,m}\mathbf J\mathbf R\mathbf V_{1,m-\alpha+1}]_{j+p_{\alpha-1},k}.
\end{align}
It is straightforward to check that
\begin{align}
\frac{\partial(\mathbf V_{\alpha+1,m}\mathbf J\mathbf R\mathbf V_{1,m-\alpha+1})}
{\partial X^{(\alpha)}_{j,k}}=
\frac1{\sqrt{p_{\alpha}}}\mathbf V_{\alpha+1,m-\alpha}\mathbf e_{k+p_{m-\alpha}}
\mathbf e_{j+p_{m-\alpha+1}}^T\mathbf V_{m-\alpha+2,m}\mathbf J\mathbf R\mathbf V_{1,m-\alpha+1}I\{\alpha\le m-\alpha\}
&\notag\\
+\frac1{\sqrt{p_{\alpha}}}\mathbf V_{[\alpha+1,m]}\mathbf J\mathbf R\mathbf V_{1,\alpha-1}\mathbf e_j\mathbf e_k^T
\mathbf V_{\alpha+1,m-\alpha+1}I\{\alpha\le m-\alpha\}
+\frac1{\sqrt{p_{\alpha}}}\mathbf V_{\alpha+1,m}\mathbf J\mathbf R\mathbf V_{1,m-\alpha}\mathbf e_{k+p_{m-\alpha}}
\mathbf e_{j+p_{m-\alpha+1}}^T&\notag\\
-\frac1{\sqrt{p_{\alpha}}}\mathbf V_{[\alpha+1,m]}\mathbf J\mathbf R\mathbf V_{1,\alpha-1}
\mathbf e_{j}\mathbf e_{k}^T\mathbf V_{\alpha+2,m}\mathbf J\mathbf R\mathbf V_{1,m-\alpha+1}&\notag\\
-\frac1{\sqrt{p_{\alpha}}}\mathbf V_{\alpha+1,m}\mathbf J\mathbf R\mathbf V_{1,m-\alpha}\mathbf e_{k+p_{m-\alpha}}\mathbf e_{j+p_{m-\alpha+1}}^T\mathbf V_{m-\alpha+2,m}\mathbf J\mathbf R\mathbf V_{1,m-\alpha+1}&.\notag
\end{align}
and
\begin{align}
\frac{\partial(\mathbf V_{\alpha+1,m}\mathbf J\mathbf R\mathbf V_{1,m-\alpha})}
{\partial X^{(m-\alpha+1)}_{j,k}}=
\frac1{\sqrt{p_{m-\alpha+1}}}\mathbf V_{\alpha+1,m-\alpha}\mathbf e_j\mathbf e_k^T\mathbf  V_{m-\alpha+2,m}\mathbf J
\mathbf R\mathbf V_{1,m-\alpha+1}I\{\alpha\le m-\alpha\}\qquad\qquad\qquad&\notag\\
+\frac1{\sqrt{p_{m-\alpha+1}}}\mathbf V_{\alpha+1,m}\mathbf J\mathbf R\mathbf V_{1,\alpha-1}
\mathbf e_{k+p_{\alpha-1}}\mathbf e_{j+p_{\alpha}}^T\mathbf V_{\alpha+1,m-\alpha+1}I\{\alpha\le m-\alpha\}
+\frac1{\sqrt{p_{\alpha}}}\mathbf V_{\alpha+1,m}\mathbf J\mathbf R\mathbf V_{1,m-\alpha}\mathbf e_j\mathbf e_k^T&\notag\\
-\frac1{\sqrt{p_{m-\alpha+1}}}\mathbf V_{\alpha+1,m}\mathbf J\mathbf R\mathbf V_{1,\alpha-1}\mathbf e_{k+p_{\alpha-1}}\mathbf e_{j+p_{\alpha}}^T\mathbf V_{\alpha+2,m}\mathbf J\mathbf R\mathbf V_{1,m-\alpha+1}&\notag\\
-\frac1{\sqrt{p_{m-\alpha+1}}}\mathbf V_{\alpha+1,m}\mathbf J\mathbf R\mathbf V_{1,m-\alpha}\mathbf e_j\mathbf e_k^T\mathbf V_{m-\alpha+2,m}\mathbf J\mathbf R\mathbf V_{1,m-\alpha+1}.&\notag
\end{align}

Applying the Lemmas \ref{teilor} and \ref{var1}, we obtain the following relation
\begin{equation}\notag
 f_{\alpha,m+1-\alpha}= D_1+\ldots+D_5+\varepsilon_n(z),
\end{equation}
where
\begin{align}
D_1=&\frac1{p_{\alpha-1}p_{\alpha}}\E\sum_{j=1}^{p_{\alpha}}
[\mathbf V_{\alpha+1,m-\alpha}]_{k,k+p_{m-\alpha}}\sum_{k=1}^{p_{\alpha-1}}
[\mathbf V_{m-\alpha+2,m}\mathbf J\mathbf R\mathbf V_{1,m-\alpha+1}]_{j+p_{m-\alpha+1},j+p_{m-\alpha+1}}
I\{\alpha\le m-\alpha\}\notag\\
D_2=&\frac1{p_{\alpha-1}p_{\alpha}}\E\sum_{j=1}^{p_{\alpha-1}}
\sum_{k=1}^{p_{\alpha}}[\mathbf V_{\alpha+1,m}\mathbf J\mathbf R\mathbf V_{1,\alpha-1}]_{k,j}
[\mathbf V_{\alpha+1,m-\alpha+1}]_{k,j+p_{m-\alpha+1}}I\{\alpha\le m-\alpha\}
\notag\\
D_3=&\frac1{{p_{\alpha}}}\sum_{k=1}^{p_{\alpha}}\E [\mathbf V_{\alpha+1,m}\mathbf J\mathbf R\mathbf V_{1,m-\alpha}]_{k,k+p_{m-\alpha}}\notag\\
D_4=&-\frac1{p_{\alpha-1}p_{\alpha}}\E\sum_{j=1}^{p_{\alpha-1}}\sum_{k=1}^{p_{\alpha}}[\mathbf V_{\alpha+1,m}\mathbf J\mathbf R\mathbf V_{1,\alpha-1}]_{k,j}
[\mathbf V_{\alpha+2,m}\mathbf J\mathbf R\mathbf V_{1,m-\alpha}]_{k,j+p_{m-\alpha+1}}\notag\\
D_5=&-\frac1{p_{\alpha-1}p_{\alpha}}\E\sum_{k=1}^{p_{\alpha}}[\mathbf V_{\alpha+1,m}\mathbf J\mathbf R\mathbf V_{1,m-\alpha}]_{k,k+p_{m-\alpha}}\notag\\ &\qquad\qquad\qquad\qquad\qquad\times\sum_{j=1}^{p_{\alpha-1}}[\mathbf V_{m-\alpha+2,m}\mathbf J\mathbf R\mathbf V_{1,m-\alpha+1}]_{j+p_{m-\alpha+1},j+p_{m-\alpha+1}}.\notag
\end{align}
 Similar we get the representation for $g_{\alpha, m-\alpha}$.
\begin{equation}\notag
 g_{\alpha,m+1-\alpha}= \overline D_1+\ldots+\overline D_5+\varepsilon_n(z),
\end{equation}
where
\begin{align}
\overline D_1=&\frac1{p_{m-\alpha+1}p_{m-\alpha}}\E\sum_{j=1}^{p_{m-\alpha}}
[\mathbf V_{\alpha+1,m-\alpha}]_{j+p_{\alpha},j}\sum_{k=1}^{p_{m-\alpha+1}}
[\mathbf V_{m-\alpha+2,m}\mathbf J\mathbf R\mathbf V_{1,m-\alpha+1}]_{k,k}I\{\alpha\le m-\alpha\}\notag\\
\overline D_2=&\frac1{p_{m-\alpha+1}p_{m-\alpha}}\E\sum_{j=1}^{p_{m-\alpha}}
\sum_{k=1}^{p_{m-\alpha+1}}[\mathbf V_{\alpha+1,m}\mathbf J\mathbf R
\mathbf V_{1,\alpha-1}]_{j+p_{\alpha},k+p_{\alpha-1}}[\mathbf V_{\alpha+1,m-\alpha+1}]_{j+p_{\alpha},k}
I\{\alpha\le m-\alpha\}
\notag\\
\overline D_3=&\frac1{p_{m-\alpha}}\E \sum_{j=1}^{p_{m-\alpha}}[\mathbf V_{\alpha+1,m}\mathbf J\mathbf R\mathbf V_{1,m-\alpha}]_{j+p_{\alpha},j}\notag\\
\overline D_4=&-\frac1{p_{m-\alpha+1}p_{m-\alpha}}\E\sum_{j=1}^{p_{m-\alpha}}\sum_{k=1}^{p_{m-\alpha+1}}[\mathbf V_{\alpha+1,m}\mathbf J\mathbf R\mathbf V_{1,\alpha-1}]_{j+p_{\alpha},k+p_{\alpha-1}}
\notag\\& \qquad\qquad\qquad\qquad\qquad\qquad\qquad\qquad\qquad\qquad\times[\mathbf V_{\alpha+2,m}\mathbf J\mathbf R\mathbf V_{1,m-\alpha+1}]_{j+p_{\alpha},k}\notag\\
\overline D_5=&-\frac1{p_{m-\alpha+1}p_{m-\alpha}}\E\sum_{j=1}^{p_{m-\alpha}}[\mathbf V_{\alpha+1,m}\mathbf J\mathbf R\mathbf V_{1,m-\alpha}]_{j+p_{\alpha},j}\sum_{k=1}^{p_{m-\alpha+1}}[\mathbf V_{m-\alpha+2,m}\mathbf J\mathbf R\mathbf V_{1,m-\alpha+1}]_{k,k}.\notag
\end{align}
\begin{lem}\label{k1}
 Under the conditions of Theorem \ref{main} there exists a 
constant $C>0$ such that the following inequality holds
\begin{align}
\max\{|D_2|,|\overline D_2|\}\le \frac C{nv}\notag
\end{align}
and
\begin{equation}\notag
 \max\{|D_4|,|\overline D_4|\}\le \frac C{nv^2}
\end{equation}
\end{lem}
\begin{proof}
 We describe the  bound for $D_2$ first.
Applying H\"older's inequality , we get
\begin{equation}\notag
 |D_2|\le \frac1{n^2}\E\|\mathbf V_{\alpha+1,m}\mathbf J\mathbf R\mathbf V_{1,\alpha-1}\|_2\|\mathbf V_{\alpha+1,m-\alpha}\|_2
\end{equation}
Applying H\"older's inequality again, we get
\begin{equation}\notag
 |D_2|\le \frac1{n^2v}\E^{\frac12}\|\prod_{\nu=1,\nu\ne\alpha}^m \mathbf H^{(\nu)}\|_2^2
\E^{\frac12}\|\mathbf V_{\alpha+1,m-\alpha}\|_2^2.
\end{equation}
Applying Lemma \ref{norm2} now, we obtain
\begin{equation}\notag
 |D_2|\le \frac C{nv}.
\end{equation}
Recall that $\|\cdot\|_2$ denotes the Frobenius norm of a matrix.
The proof of the  bound for $\overline D_2$, $D_4$ and $\overline D_4$ are similar.
Thus the  Lemma is proved.
\end{proof}

\begin{lem}\label{k2}
 Under the conditions of Theorem \ref{main} there exists a 
constant $C>0$ such that the following inequality holds
\begin{align}
\max\{|D_1|,\,|\overline D_1|\}\le \frac C{nv}.\notag
\end{align}
\end{lem}
\begin{proof}
 Applying H\"older's inequality and Lemma \ref{var1}, we get
\begin{align}
 |D_1-\frac1{p_{\alpha-1}}&\E\sum_{j=1}^{p_{\alpha-1}}[\mathbf V_{\alpha+1,m-\alpha}]_{k,k+p_{m-\alpha}}\notag\\&\times\frac1{p_{\alpha}}\E\sum_{k=1}^{p_{\alpha}}[\mathbf V_{m-\alpha+2,m}\mathbf J\mathbf R\mathbf V_{1,m-\alpha}]_{j+p_{m-\alpha+1},j+p_{m-\alpha+1}}|\le \frac C{nv}.\notag
\end{align}
Thus the Lemma is proved.
\end{proof}
\begin{lem}\label{k3}
  Under the conditions of Theorem \ref{main} there exists a 
constant $C>0$ that the following inequality holds
\begin{align}
&\left|D_5+\frac1{p_{\alpha}}\E\sum_{k=1}^{p_{\alpha}}[\mathbf V_{\alpha+1,m}\mathbf J\mathbf R\mathbf V_{1,m-\alpha}]_{k,k+p_{m-\alpha}}\right.
\notag\\ &\qquad\quad\qquad\qquad\left.\times\frac1{p_{\alpha-1}}
\E\sum_{j=1}^{p_{\alpha-1}}[\mathbf V_{m-\alpha+2,m}\mathbf J\mathbf R\mathbf V_{1,m-\alpha+1}]_{j+p_{m-\alpha+1},j+p_{m-\alpha+1}}\right|\le \frac C{nv^2}.\notag
\end{align}
and
\begin{align}
 &\left|\overline D_5+\frac1{p_{m-\alpha}}\E\sum_{j=1}^{p_{m-\alpha}}[\mathbf V_{\alpha+1,m}\mathbf J\mathbf R\mathbf V_{1,m-\alpha}]_{j+p_{m-\alpha},j}\right.
\notag\\ &\qquad\qquad\qquad\quad\left.\times\frac1{p_{m-\alpha+1}}
\E\sum_{k=1}^{p_{m-\alpha+1}}[\mathbf V_{m-\alpha+2,m}\mathbf J\mathbf R\mathbf V_{1,m-\alpha+1}]_{k,k}\right|\le \frac C{nv^2}.\notag
\end{align}

\end{lem}
\begin{proof}

 Applying H\"older's inequality and Lemma \ref{var1}, we conclude the result.
\end{proof}
Using the obvious  equality $\Tr\mathbf A\mathbf B=\Tr \mathbf B\mathbf A$, it id straightforward to check that
\begin{align}
 \frac1{p_{\alpha-1}}\sum_{j=1}^{p_{\alpha-1}}\E[\mathbf V_{m-\alpha+2,m}\mathbf J\mathbf R\mathbf V_{1,m-\alpha+1}]_{j+p_{m-\alpha+1},j+p_{m-\alpha+1}}\notag\\=
\frac1{p_{\alpha-1}}\sum_{j=1}^{p_m}\E[\mathbf V\mathbf J\mathbf R]_{j+n,j+n}&\notag={y_{\alpha-1}}(1+zs_n(z)).\notag
\end{align}
This implies that
\begin{equation}\notag
 1-\frac1{p_{\alpha-1}}\sum_{j=1}^{p_{\alpha-1}}\E[\mathbf V_{m-\alpha+2,m}\mathbf J\mathbf R
\mathbf V_{1,m-\alpha+1}]_{j+p_{m-\alpha+1},j+p_{m-\alpha+1}}=(1-y_{\alpha-1}-zy_{\alpha-1}s_n(z))
\end{equation}

Lemmas \ref{k1}--\ref{k3} and last equality together imply
\begin{equation}\label{k4}
 f_{\alpha,m-\alpha+1}=-(1-y_{\alpha-1}-zy_{\alpha-1}s(z))f_{\alpha+1,m-\alpha}+\varepsilon_n(z).
\end{equation}
Similar we show that
\begin{equation}\label{k5}
 g_{\alpha,m-\alpha+1}=-(1-y_{m-\alpha+1}-zy_{m-\alpha+1}s_n(z))g_{\alpha+1,m-\alpha}+\varepsilon_n(z).
\end{equation}
Note that
\begin{equation}\label{k6}
 f_{m+1,0}=\frac1{p_m}\sum_{k=1}^{p_m} \E[\mathbf J\mathbf R]_{k,k+p_{0}}=
\frac n{p_m}\frac1n\sum_{k=1}^{p_m}\E[\mathbf R]_{k+n,k+n}=-\frac1{z}(1-y_m-zy_ms_n(z))
\end{equation}
and
\begin{equation}\label{k7}
 g_{m+1,0}=\frac1{p_0}\sum_{k=1}^{n} \E[\mathbf J\mathbf R]_{k+p_m,k}=\frac1n\sum_{j=1}^n\E[\mathbf R]_{jj}=s_n(z)
\end{equation}
Equalities  (\ref{k4})--(\ref{k7}) together imply
\begin{align}\label{k8}
 f_{2,m}=(-1)^{m+1}\frac1z\prod_{q=1}^m{(1-y_q-zy_qs_n(z))}+\varepsilon_n(z)\notag\\
g_{2,m}=(-1)^{m+1}\frac{s_n(z)}z\prod_{q=1}^{m-1}({1-y_q-zy_qs_n(z)})+\varepsilon_n(z).
\end{align}

Equalities  (\ref{a1}) and  (\ref{k8})  together imply
\begin{equation}\notag
 1+zs_n(z)=(-1)^{m+1}\frac{s_n(z)}z\prod_{q=1}^m{(1-y_q-zy_qs_n(z))}+\varepsilon_n(z).\notag
\end{equation}
We rewrite the last equation as follows
\begin{equation}\label{l1}
 1+zs_n(z)+(-1)^{m}\frac{s_n(z)}z\prod_{q=1}^m{(1-y_q-zy_qs_n(z))}=\varepsilon_n(z).
\end{equation}
Let Stieltjes transform $s(z)$ satisfies the equation
\begin{equation}\label{l2}
 1+zs(z)+(-1)^{m}\frac{s(z)}z\prod_{q=1}^m(1-y_q-zy_qs(z))=0\notag
\end{equation}
Introduce the notations
\begin{align}
 Q_0&:=\frac1z\prod_{q=1}^m(1-y_q-zy_qs_n(z)),\notag\\ Q_{\nu}&:=s(z)\prod_{q=1}^{\nu-1}
 (1-y_q-zy_qs(z))\prod_{q=\nu+1}^{m}(1-y_q-zy_qs_n(z)).\notag
\end{align}

Relations (\ref{l1}) and (\ref{l2}) together imply that, for
\begin{equation}\label{l3}
 s_n(z)-s(z)= \frac{\varepsilon_n(z)}{z+(-1)^{m-1}\sum_{q=0}^m Q_q}
\end{equation}
Note that
\begin{equation}\notag
 \max\{|zs(z)|,\,|zs_n(z)|\}\le 1+\frac1v
\end{equation}
and
\begin{equation}\notag
 \max\{|s_n(z)|,\,|s(z)|\}\le \frac1v
\end{equation}
Applying these inequalities, we obtain
\begin{equation}\notag
|Q_q| \le \frac 1v(1+\frac1v)^m.
\end{equation}
We may choose $v\ge m+1$. Then $\frac1v(1+\frac1v)^m\le \frac ev$. If we choose $v$ such that
 $\frac ev\le \frac v{2m}$, we get
\begin{equation}\notag
 |z+(-1)^{m-1}\sum_{q=0}^mQ_q|\ge \frac v2.
\end{equation}
This implies that, for $v\ge V_1:=\sqrt{\frac{2m}{\rm e}}$,
\begin{equation}\label{l5}
 |s_n(z)-s(z)|\le \frac{C|\varepsilon_n(z)|}{v}\le C\tau_n
\end{equation}

From inequality (\ref{l5}) we conclude that there exists 
 an open set with non-empty interior where  $s_n(z)$ converges
 to $s(z)$.
 The Stieltjes transform of our {random matrices} 
 is an analytic function on $\mathcal C^+$ and locally bounded ($|s_n(z)|\le v^{-1}$ for any $v>0$). By Montel's Theorem (see, for instance,
\cite{Conway}, p. 153, Theorem 2.9) $s_n(z)$ converges  to $s(z)$ uniformly on
any compact set $\mathcal K\subset \mathcal C^+$  in the upper half-plane.
  This implies that $\Delta_n\to 0$ as $n\to\infty$. Thus the proof 
 of Theorem \ref{main} in the general case is complete.

\section{Appendix}

\begin{lem}\label{mean}
 Under the conditions of Theorem \ref{main} we have, for any $j,k=1,\ldots,p_{\alpha-1}+p_{\beta}$,
 and for any $1\le \alpha\le \beta\le m$,
\begin{equation}\notag
\E[\mathbf V_{\alpha,\beta}]_{jk}=0
\end{equation}
\end{lem}
\begin{proof}For $\alpha=\beta$ the claim is easy. Let $\alpha<\beta$. We consider the case $j=1,\ldots,p_{\alpha-1}$ and
$k=1,\ldots,p_{\beta}$ only. The other cases are similar.
 Direct calculations show that
\begin{equation}\notag
 \E[\mathbf V_{\alpha,\beta}]_{jk}=\frac1{n^\frac{\beta-\alpha}2}\sum_{j_1=1}^{p_{\alpha}}
\sum_{j_2=1}^{p_{\alpha+1}}\dots \sum_{j_{\beta-\alpha}=1}^{p_{\beta-1}}
\E X^{(\alpha)}_{j,j_1}X^{(\alpha+1)}_{j_1,j_2}\cdots X^{((\beta))}_{j_{\beta-\alpha},k}=0
\end{equation}
Thus the Lemma is proved.
\end{proof}
In all Lemmas below we shall assume that
\begin{equation}\label{as1}
 \E X_{jk}^{(\nu)}=0,\quad\E|X_{jk}^{(\nu)}|^2=1, \quad |X_{jk}^{(\nu)}|\le c\tau_n\sqrt n\quad\text{a. s.}
\end{equation}

\begin{lem}\label{norm2}
 Under the conditions of Theorem \ref{main} assuming (\ref{as1}), we have,  for any $1\le \alpha\le \beta\le m$,
\begin{equation}
\E\|\mathbf V_{\alpha,\beta}\|_2^2\le Cn
\end{equation}

\end{lem}
\begin{proof}We shall consider the case $\alpha<\beta$ only. The other cases
 are obvious.
 Direct calculation shows that
\begin{equation}\notag
 \E\|\mathbf V_{\alpha,\beta}\|_2^2\le\frac C{n^{\beta-\alpha+1}}\sum_{j=1}^n\sum_{j_1=1}^{p_{\alpha-1}}
\sum_{j_2=1}^{p_{\alpha}}\dots \sum_{j_{\beta-\alpha}=1}^{p_{\beta-1}}\sum_{k=1}^{p_{\beta}}
\E [X^{(\alpha)}_{j,j_1}X^{(\alpha+1)}_{j_1,j_2}\cdots X^{(\beta)}_{j_{\beta-\alpha},k}]^2
\end{equation}
By independents of random variables, we get
\begin{equation}\notag
\E\|\mathbf V_{\alpha,\beta}\|_2^2\le  Cn
\end{equation}
Thus the Lemma is proved.
\end{proof}

\begin{lem}\label{norm4}
 Under the condition of Theorem \ref{main} and assumption (\ref{as1}) 
we have, for any $j=1,\ldots p_{\alpha-1}$, 
$k=1\ldots p_{\beta}$ an $r\ge1$,
\begin{equation}
 \E\|\mathbf V_{\alpha,\beta}\mathbf e_k\|_2^{2r}\le C_r,\quad
\E\|\mathbf V_{\alpha,\beta}\mathbf e_{j+p_{\beta}}\|_2^{2r}\le C_r
\end{equation}
and
\begin{equation}
 \E\|\mathbf e_j^T\mathbf V_{\alpha,\beta}\|_2^{2r}\le C_r,\quad
\E\|\mathbf e_{k+p_{\alpha-1}}^T\mathbf V_{\alpha,\beta}\|_2^{2r}\le C_r,
\end{equation}
with some positive constant $C_r$ depending on $r$.
\end{lem}
\begin{proof}
 By definition  of the  matrices $\mathbf V_{\alpha,\beta}$, we may write
\begin{equation}
 \|\mathbf e_j^T\mathbf V_{\alpha,\beta}\|_2^{2}=\frac1{p_{\alpha}\cdots p_{\beta}}
\sum_{l=1}^{p_{\beta}}\left|\sum_{j_{\alpha}=1}^{p_{\alpha}}\cdots\sum_{j_{\beta-1}=1}^{p_{\beta-1}}
X_{jj_{\alpha}}^{(\alpha)}\cdots X_{j_{\beta-1}l}^{(\beta)}\right|^2
\end{equation}
Using this representation, we get
\begin{equation}\label{mom10}
 \E\|\mathbf e_j^T\mathbf V_{\alpha,\beta}\|_2^{2r}=\frac1{p_{\alpha-1}^r\cdots p_{\beta-1}^r}\sum_{l_1=1}^{p_{\beta}}\cdots
\sum_{l_r=1}^{p_{\beta}}\E\prod_{q=1}^r
\left(\sum_{j_{\alpha}=1}^{p_{\alpha}}\cdots\sum_{j_{\beta-1}=1}^{p_{\beta-1}}
\sum_{\widehat j_{\alpha}=1}^{p_{\alpha}}\cdots
\sum_{\widehat j_{\beta-1}=1}^{p_{\beta-1}}A^{(l_q)}_{(j_{\alpha},\ldots,j_{\beta-1},\widehat j_{\alpha},\ldots,
\widehat j_{\beta-1})}\right)
\end{equation}
where
\begin{equation}\label{per}
A^{(l_q)}_{(j_{\alpha},\ldots,j_{\beta-1},\widehat j_1,\ldots,\widehat j_{\beta-1})}=X_{jj_{\alpha}}^{(\alpha)}
\overline X_{j \widehat j_{\alpha}}^{(\alpha)}X_{j_{\alpha}j_{\alpha+1}}^{(\alpha+1)}
\overline X_{{\widehat j}_{\alpha} \widehat j_{\alpha+1}}^{(\alpha+1)}\cdots X_{j_{\beta-2}j_{\beta-1}}^{(\beta-1)} 
\overline X_{\widehat j_{\beta-2}\widehat  j_{\beta-1}}^{(\beta-1)}X_{j_{\beta-1}l_q}^{(\beta)}
\overline X_{\widehat j_{\beta-1}l_q}^{(\beta)}.
\end{equation}
By $\overline x$ we denote the complex conjugate of  $x$.
{Rewriting} the product on  the r.h.s of (\ref{mom10}), we get
\begin{align}
 \E\|\mathbf e_j^T\mathbf V_{\alpha,\beta}\|_2^{2r}=\frac1{p_{\alpha-1}^r\cdots p_{\beta-1}^r}
{\sum}^{**}\E\prod_{q=1}^rA^{(l_q)}_{(j_{\alpha}^{(q)},\ldots,j_{\beta-1}^{(q)},
{\widehat j}_1^{(\nu)},\ldots,{\widehat j}_{\beta-1}^{(q)})},
\end{align}
where ${\sum}^{**}$ is taken  over all set of indices $j_{\alpha}^{(q)},\ldots, j_{\beta-1}^{(q)}, l_q$ and
${\widehat j}_{\alpha}^{(\nu)},\ldots,{\widehat j}_{\beta-1}^{(q)}$ where
$j_k^{(q)},{\widehat j}_k^{(q)}=1,\ldots,p_k$, $k=\alpha,\ldots,\beta-1$, $l_q=1,\ldots,p_{\beta}$ and  $q=1,\ldots,r$.
Note that the summands in the right hand side of (\ref{per}) is equal 0 if there is at least one term in the product \ref{per}
which appears only one time. This implies that  the summands in the  right hand side of (\ref{per}) 
is not equal zero only if the union of all sets of indices in r.h.s of (\ref{per}) consist from at least $r$
 different terms and each term appears at least twice.

Introduce the random variables, for $\nu=\alpha+1,\ldots, \beta-1$,
\begin{align}
 \zeta^{(\nu)}_{j^{(1)}_{\nu-1},\ldots,j^{(r)}_{\nu-1},j^{(1)}_{\nu},\ldots,j^{(r)}_{\nu},
{\widehat j}^{(1)}_{\nu-1},\ldots,{\widehat j}^{(r)}_{\nu-1},{\widehat j}^{(1)}_{\nu},\ldots,
{\widehat j}^{(r)}_{\nu}}&
=X^{(\nu)}_{j^{(1)}_{\nu-1},j^{(1)}_{\nu}}
\cdots X^{(\nu)}_{j^{(r)}_{\nu-1},j^{(r)}_{\nu}}
{\overline X}^{(\nu)}_{{\widehat j}^{(1)}_{\nu-1},{\widehat j}^{(1)}_{\nu}},
\cdots {\overline X}^{(\nu)}_{{\widehat j}^{(r)}_{\nu-1},{\widehat j}^{(r)}_{\nu}},
\end{align}
and
\begin{align}
\zeta^{(\alpha)}_{j^{(1)}_{1},\ldots,j^{(r)}_{1},
{\widehat j}^{(1)}_{1},\ldots,{\widehat j}^{(r)}_{1}}&=X^{(\alpha)}_{jj_1^{(\alpha)}}
\cdots X^{(\alpha)}_{j^{(r)}_{a}j^{(r)}_{a+1}}
{\overline X}^{(\alpha)}_{j{\widehat j}^{(1)}_{a}}
\cdots {\overline X}^{(\alpha)}_{{\widehat j}^{(r)}_{a},{\widehat j}^{(r)}_{a+1}}
\notag\\
\zeta^{(\beta)}_{j^{(1)}_{\beta-1},\ldots,j^{(r)}_{\beta-1},
{\widehat j}^{(1)}_{\beta-1},\ldots,{\widehat j}^{(r)}_{\beta-1},l_q}
&=X^{(\beta)}_{j^{(1)}_{\beta-1}j^{(1)}_{\beta}}
\cdots X^{(\beta)}_{j^{(r)}_{\beta-1}l_q}
{\overline X}^{(\beta)}_{{\widehat j}^{(1)}_{\beta-1},l_q},
\cdots {\overline X}^{(\beta)}_{{\widehat j}^{(r)}_{\beta-1},l_q}.
\notag
\end{align}
 Assume that the set of indices $j^{(1)}_{\alpha},\ldots,j^{(r)}_{\alpha},
{\widehat j}^{(1)}_{\alpha},\ldots,{\widehat j}^{(r)}_{\alpha}$ contains $t_{\alpha}$ different indexes, say
$i_1^{(\alpha)},\ldots,i_{t_{\alpha}}^{(\alpha)}$ with multiplicities $k_1^{(\alpha)},\ldots,k_{t_{\alpha}}^{(\alpha)}$
respectively, $k_1^{(\alpha)}+\ldots+k_{t_{\alpha}}^{(\alpha)}=2r$.
Note that $\min\{k_1^{(\alpha)},\ldots,k_{t_{\alpha}}^{(\alpha)}\}\ge 2$. Otherwise,\newline $|\E\zeta^{(\alpha)}_{j^{(1)}_{a},\ldots,j^{(r)}_{a},
{\widehat j}^{(1)}_{\alpha},\ldots,{\widehat j}^{(r)}_{\alpha}}|=0$. By assumption (\ref{as1}), we have
\begin{equation}\label{mom1}
|\E\zeta^{(\alpha)}_{j^{(1)}_{\alpha},\ldots,j^{(r)}_{\alpha},
{\widehat j}^{(1)}_{\alpha},\ldots,{\widehat j}^{(r)}_{\alpha}}|\le C(\tau_n\sqrt n)^{2r-2t_{\alpha}} 
\end{equation}
Similar bounds we get for $|\E
\zeta^{(\beta)}_{j^{(1)}_{\beta-1},\ldots,j^{(r)}_{1},
{\widehat j}^{(1)}_{\beta-1},\ldots,{\widehat
  j}^{(r)}_{\beta-1},l_q}|$. Assume that
the set of indexes $\{j^{(1)}_{\beta-1},\ldots,j^{(r)}_{\beta-1}$,
${\widehat j}^{(1)}_{\beta-1},\ldots,{\widehat j}^{(r)}_{\beta-1}\}$ contains $t_{\beta-1}$ different indices, say,
$i_1^{(\beta-1)},\ldots,i_{t_{\beta-1}}^{(\alpha)}$ with multiplicities\newline 
$k_1^{(\beta-1)},\ldots,k_{t_{\beta-1}}^{(\alpha)}$
respectively, $k_1^{(\beta-1)}+\ldots+k_{t_{\beta-1}}^{(\alpha)}=2r$. Then
\begin{equation}\label{mom2}
|\E
\zeta^{(\beta)}_{j^{(1)}_{\beta-1},\ldots,j^{(r)}_{1},
{\widehat j}^{(1)}_{\beta-1},\ldots,{\widehat j}^{(r)}_{\beta-1},l_q}|\le C(\tau_n\sqrt n)^{2r-2t_{\beta-1}} 
\end{equation}
Furthermore, assume that for $\alpha+1\le \nu\le \beta-2$ there are $t_{\nu}$
{different  pairs of indices}, say, $(i_{\alpha},
i'_{\alpha}),
\ldots(i_{t_{\beta}},i'_{t_{\beta}})$ in the set\newline $\{j^{(1)}_{\alpha},\ldots,j^{(r)}_{\alpha},
{\widehat j}^{(1)}_{\alpha},\ldots,{\widehat j}^{(r)}_{\alpha},\ldots,j^{(1)}_{\beta-1},\ldots,j^{(r)}_{\beta-1},
{\widehat j}^{(1)}_{\beta-1},\ldots,{\widehat j}^{(r)}_{\beta-1},l_1,l_r\}$
with multiplicities\newline$k_1^{(\nu)},\ldots,k_{t_{\nu}}^{(\nu)}$.
Note that
\begin{equation}
 k_1^{(\nu)}+\ldots+k_{t_{\nu}}^{(\nu)}=2r
\end{equation}
and
\begin{equation}\label{mom3}
|\E\zeta^{(\nu)}_{j^{(1)}_{\nu-1},\ldots,j^{(r)}_{\nu-1},j^{(1)}_{\nu},\ldots,j^{(r)}_{\nu},
{\widehat j}^{(1)}_{\nu-1},\ldots,{\widehat j}^{(r)}_{\nu-1},{\widehat j}^{(1)}_{\nu},\ldots,{\widehat j}^{(r)}_{\nu}}|\le C(\tau_n\sqrt n)^{2r-2t_{\nu}}.
\end{equation}
Inequalities (\ref{mom1})-(\ref{mom3}) together yield
\begin{equation}\label{mom11}
 |\E\prod_{q=1}^rA^{(l_q)}_{(j_{\alpha}^{(q)},\ldots,
j_{\beta-1}^{(q)},{\widehat j}_1^{(q)},\ldots,{\widehat j}_{\beta-1}^{(q)})}|
\le C(\tau_n\sqrt n)^{2r(\beta-\alpha)-2(t_1+\ldots+t_{\beta-\alpha})}.
\end{equation}
It is straightforward to check that the  number $\mathcal N(t_{\alpha},\ldots,t_{\beta})$ of sequences of indices
\newline$\{j^{(1)}_{\alpha},\ldots,j^{(r)}_{\alpha},
{\widehat j}^{(1)}_{\alpha},\ldots,{\widehat j}^{(r)}_{\alpha},\ldots,j^{(1)}_{\beta-1},\ldots,j^{(r)}_{\beta-1},
{\widehat j}^{(1)}_{\beta-1},\ldots,{\widehat j}^{(r)}_{\beta-1},l_1,\ldots,l_r\}$ with $t_{\alpha},\ldots,t_{\beta}$
{of different pairs}  satisfies the inequality
\begin{equation}\label{finish}
 \mathcal N(t_{\alpha},\ldots,t_{\beta})\le Cn^{t_{\alpha}+\ldots+t_{\beta}},
\end{equation}
with $1\le t_i\le r,\quad i=\alpha,\ldots,\beta$.
By the assumption of Theorem \ref{main}, we have
\begin{equation}\label{99} 
 cn\le p_{\nu}\le Cn
\end{equation}
 for any $\nu=1,\ldots,m$.
Note that in the case $t_{\alpha}=\cdots=t_b=r$ the inequalities (\ref{mom1})--(\ref{mom3}) imply
\begin{equation}\label{mom4}
 \E\zeta^{(\nu)}_{j^{(1)}_{\nu-1},\ldots,j^{(r)}_{\nu-1},j^{(1)}_{\nu},\ldots,j^{(r)}_{\nu},
{\widehat j}^{(1)}_{\nu-1},\ldots,{\widehat j}^{(r)}_{\nu-1},{\widehat j}^{(1)}_{\nu},\ldots,{\widehat j}^{(r)}_{\nu}}\le C
\end{equation}
Inequalities  (\ref{finish}),  (\ref{mom11}), (\ref{mom4}), and the 
representation (\ref{mom10}) 
together conclude the proof.
\end{proof}




\begin{lem}\label{var0}Under the  conditions of Theorem \ref{main} assuming (\ref{as1}), we have
\begin{equation}\notag
\E|\frac1n(\Tr \mathbf R-\E\Tr \mathbf R)|\le \frac C{nv^2}.
\end{equation}
\end{lem}
\begin{proof}Consider the matrix $\mathbf X^{(\nu,j)}$ obtained from the 
matrix $\mathbf
X^{(\nu)}$  by replacing  the  $j$-th row by a row with zero-entries. 
We define the following matrices
\begin{equation} \notag\mathbf H^{(\nu,j)}=\mathbf H^{(\nu)}-\mathbf e_j\mathbf
e_j^T\mathbf H^{(\nu)},
\end{equation}
and
\begin{equation}\notag
{\widetilde{\mathbf H}}^{(m-\nu+1,j)}={{\mathbf
H}}^{(m-\nu+1)}-{{\mathbf H}}^{(m-\nu+1)}\mathbf
e_{j+p_{m-\nu+1}}\mathbf e_{j+p_{m-\nu+1}}^T.
\end{equation}
For {simplicity}  we shall assume that  $\nu\le m-\nu+1$. Define $$\mathbf
V^{(\nu,j)}=\prod_{q=1}^{\nu-1}\mathbf H^{(q)}\,\mathbf
H^{(\nu,j)}\prod_{q=\nu+1}^{m-\nu}\mathbf H^{(q)}{\widetilde{\mathbf
H}}^{(m-\nu+1,j)}\prod_{q=m-\nu+2}^{m}\mathbf H^{(q)}.$$ We shall use
the following inequality. For any Hermitian matrix $\mathbf A$ and $\mathbf
B$ with spectral distribution function $F_A(x)$ and $F_B(x)$
respectively, we have
\begin{equation}\label{trace}
|\Tr (\mathbf A-z\mathbf I)^{-1}-\Tr (\mathbf B-z\mathbf I)^{-1}|\le
\frac {\text{\rm rank}(\mathbf A-\mathbf B)}{v}.
\end{equation}
It is straightforward to show that
\begin{equation}\label{rank}
\text{\rm rank}(\mathbf V\mathbf J-\mathbf V^{(\nu,j)}\mathbf J)\le 4m.
\end{equation}
Inequality (\ref{trace}) and (\ref{rank}) together imply
\begin{equation}\notag
|\frac1{2n}(\Tr \mathbf R-\Tr \mathbf R^{(\nu,j)})|\le \frac C{nv}.
\end{equation}

We may now apply a standard martingale expansion
technique already used in Girko \cite{Girko:89}. We may introduce
$\sigma$-algebras $\mathcal F_{\nu,j}=\sigma\{X^{(\nu)}_{lk},\,
j< l\le p_{\nu-1}, k=1,\ldots,p_{\nu}; X^{(q)}_{pk}$,
$q=\nu+1,\ldots m, \,p=1,\ldots,p_{q-1},k=1,\ldots,p_{q}\}$ and 
use the representation
\begin{equation}\notag
\Tr\mathbf R-\E\Tr\mathbf R=\sum_{\nu=1}^m\sum_{j=1}^{p_{\nu-1}}(\E_{\nu,j-1}\Tr\mathbf R-\E_{\nu,j}\Tr\mathbf R),
\end{equation}
where $\E_{\nu,j}$ denotes the {conditional expectation with respect to the} 
 $\sigma$-algebra $\mathcal F_{\nu,j}$. Note that $\mathcal F_{\nu,p_{\nu-1}}=\mathcal F_{\nu+1,0}$
\end{proof}
\begin{lem}\label{var1}
Under the conditions of Theorem \ref{main} we have, for $1\le a\le m$,
\begin{equation}\notag
\E|\frac1n(\sum_{k=1}^{p_{m-a}}[\mathbf V_{a+1,m}\mathbf J\mathbf R\mathbf
V_{1,m-a}]_{k,k+p_{\alpha}}-\E\sum_{k=1}^{p_{m-\alpha}}[\mathbf V_{\alpha+1,m}\mathbf J\mathbf R\mathbf V_{1,m-\alpha}]_{kk
+{p_{\alpha}}})|^2\le \frac C{n v^4}.
\end{equation}
and, for $1\le \alpha\le m-1$,
\begin{equation}\notag
\E\left|\frac1n\left(\sum_{k=1}^{p_{m-\alpha+1}}[\mathbf V_{m-\alpha+2,m}\mathbf J\mathbf R\mathbf
V_{1,m-\alpha+1}]_{k,k}-\E\sum_{j=1}^{p_{m-\alpha+1}}[\mathbf V_{m-\alpha+2,m}\mathbf J
\mathbf R\mathbf V_{1,m-\alpha+1}]_{kk}\right)\right|^2\le \frac C{n v^4}.
\end{equation}
\end{lem}
\begin{proof}We prove the first inequality only. The proof of other one 
 is similar.
We introduce the folowing  matrices, for $\nu=1,\ldots,m$ and for
$j=1,\ldots,p_{\nu-1}$,
 $\mathbf X^{(\nu,j)}=\mathbf X^{(\nu)}-\mathbf
e_j\mathbf e_j^T\mathbf X^{(\nu)}$,
 and $\mathbf H^{(\nu,j)}=\mathbf
H^{(\nu)}-\mathbf e_j\mathbf e_j^T\mathbf H^{(\nu)}$ and \newline
${\widetilde{\mathbf H}}^{(m-\nu+1,j)}=\mathbf H^{(m-\nu+1,j)}-\mathbf
H^{(m-\nu+1)}\mathbf e_{j+p_{m-\nu+1}} \mathbf e_{j+p_{m-\nu+1}}^T$.
 Note that the matrix $\mathbf X^{(\nu,j)}$ is
obtained from matrix $\mathbf X^{(\nu)}$ by replacing all entries of the
$j$-th row by $0$. Similar to the proof  of the previous Lemma we introduce
matrices $\mathbf V^{(\nu,j)}_{c,d}$ by replacing {in the  definition of
  the  matrix} $\mathbf V_{c,d}$ the matrix $\mathbf H^{(\nu)}$ by $\mathbf
H^{(\nu,j)}$ and the matrix $\mathbf H^{(m-\nu+1)}$ by ${\widetilde{\mathbf
H}}^{(m-\nu+1,j)}$. For instance, for $c\le\nu\le m-\nu+1\le d$, we get
\begin{equation}\notag
\mathbf V^{(\nu,j)}_{c,d}=\prod_{q=a}^{\nu-1}\mathbf H^{(q)}\,\mathbf
H^{(\nu,j)}\prod_{q=\nu+1}^{m-\nu}\mathbf H^{(q)}{\widetilde{\mathbf
H}}^{(m-\nu+1,j)}\prod_{q=m-\nu+1}^{b}\mathbf H^{(q)}
\end{equation}.

 Define as well $\mathbf
V^{(\nu,j)}:= \mathbf V_{1,m}^{(\nu,j)}$ and
$\mathbf R^{(j)}:= (\mathbf V^{(\nu,j)}-z\mathbf I)^{-1}$. Consider the
following quantities, for $\nu=1\ldots,m$ and $j=1,\ldots,p_{\nu-1}$,
\begin{equation}\notag
\Xi_j:=\sum_{k=1}^{p_{m-\alpha}}[\mathbf V_{\alpha+1,m}\mathbf J\mathbf R\mathbf
V_{1,m-\alpha+1}]_{kk+p_{a}}- \sum_{k=1}^n[\mathbf V^{(\nu,j)}_{\alpha+1,m}\mathbf J\mathbf
R^{(\nu,j)}\mathbf V^{(\nu,j)}_{1,m-\alpha+1}]_{kk+p_{\alpha}}
\end{equation}
We represent it in the following form
\begin{equation}\notag
\Xi_j:= \Xi_j^{(1)}+\Xi_j^{(2)}+\Xi_j^{(3)},
\end{equation}
where
\begin{align}
\Xi_{\nu,j}^{(1)}&= =\sum_{k=1}^{p_{m-\alpha}}[(\mathbf V_{\alpha+1,m}
-\mathbf V_{\alpha+1,m}^{(\nu,j)})\mathbf J\mathbf R\mathbf  V_{1,m-\alpha+1}]_{k,k+p_{\alpha}},\notag\\
\Xi_{\nu,j}^{(2)}&= \sum_{k=1}^{p_{m-\alpha}}[\mathbf V_{\alpha+1,m}^{(\nu,j)}\mathbf J(\mathbf R-\mathbf R^{(\nu,j)})
\mathbf J\mathbf  V_{1,m-\alpha+1}]_{kk+p_{\alpha}},\notag\\
\Xi_{\nu,j}^{(3)}&= \sum_{k=1}^{p_{m-\alpha}}[\mathbf
V^{(j)}_{\alpha+1,m}\mathbf J\mathbf R^{(\nu,j)}(\mathbf  V_{1,m-\alpha+1}-\mathbf
V_{1,m-\alpha+1}^{(\nu,j)})]_{kk+p_{\alpha}}.\notag
\end{align}
Note that
\begin{align}
 \mathbf V_{a+1,m}-\mathbf V^{(\nu,j)}_{a+1,m}&=\mathbf V_{a+1,\nu-1}(\mathbf H^{(\nu)}-\mathbf
 H^{(\nu,j)})\mathbf V_{\nu+1,m}\notag\\&+
 \mathbf V_{a+1,\nu-1}\mathbf H^{(\nu,j)}\mathbf
 V_{\nu+1,m-\nu}(\widetilde {\mathbf H}_{m-\nu+1}-{\widetilde {\mathbf
 H}}_{m-\nu+1}^{(\nu,j)})\mathbf V_{m-\nu+2,m}.\notag
\end{align}
By definition of the  matrices $\mathbf H^{\nu,j}$ and ${\widetilde{\mathbf
H}}^{m-\nu+1,j}$, we have
\begin{align}
\sum_{k=1}^{p_{m-a}}[(\mathbf V_{a+1,m} -\mathbf
V_{a+1,m}^{(\nu,j)})\mathbf J\mathbf R\mathbf
V_{1,m-\nu+1}]_{k,k+p_{a}}=[\mathbf V_{\nu+1,m}\mathbf J\mathbf R\mathbf V_{1,m-a+1}\mathbf{\widetilde J}\mathbf V_{a+1,\nu}]_{j,j}&\notag\\
+[\mathbf V_{m-\nu+2,m}\mathbf J\mathbf R\mathbf V_{1,m-a+1}\mathbf{\widetilde
J}\mathbf V_{a+1,m-a+1}]_{j+p_{\nu-1},j+p_{\nu-1}},&\notag
\end{align}
where
$$
\mathbf{\widetilde J}=\left(\begin{matrix}{\mathbf O\quad\mathbf
I}\\{\mathbf O\quad\mathbf O}\end{matrix}\right)
$$

This equality implies that
\begin{align}
 |\Xi_j^{(1)}|&\le |[\mathbf V_{\nu+1,m}\mathbf J\mathbf R\mathbf V_{1,m-a+1}\mathbf{\widetilde J}\mathbf V_{a+1,\nu}]_{j,j+n}|\notag\\&\qquad\qquad+|[\mathbf V_{m-\nu+2,m}\mathbf J\mathbf R\mathbf V_{1,m-a+1}\mathbf{\widetilde
J}\mathbf V_{a+1,m-\nu+1}]_{j+p_{\nu-1},j+p_{\nu-1}}|.\notag
\end{align}

Using the obvious inequality $\sum_{j=1}^n a_{jj}^2\le \|\mathbf A\|_2^2$ for any matrix $\mathbf A=(\alpha_{jk})$,
$j,k=1,\ldots,n$, we get
\begin{align}
T_1:=\sum_{j=1}^n\E|\Xi_j^{(1)}|^2\le &\E\|\mathbf V_{\nu+1,m}\mathbf J\mathbf R\mathbf V_{1,m-a+1}\mathbf{\widetilde J}\mathbf V_{a+1,\nu}\|_2^2\notag\\&+\E\|\mathbf V_{m-\nu+2,m}\mathbf J\mathbf R\mathbf V_{1,m-a+1}\mathbf{\widetilde
J}\mathbf V_{a+1,m-\nu+1}\|_2^2.\notag
\end{align}

By Lemma \ref{norm2}, we get
\begin{equation}\label{T1}
 T_1\le \frac{C}{v^2}\E\|\mathbf V_{a+1,m}\mathbf V_{1,m-a+1}\|_2^2\le \frac {Cn}{v^2}
\end{equation}

Consider now
\begin{equation}\notag
 T_2=\sum_{j=1}^n\E|\Xi_j^{(2)}|^2.
\end{equation}
Using that $\mathbf R-\mathbf R^{(j)}=-\mathbf R^{(j)}(\mathbf V-\mathbf
V^{(\nu,j)})\mathbf R$, we get
\begin{align}
 |\Xi_{j}^{(2)}|&\le |\sum_{k=1}^{p_{a-1}}[\mathbf V^{(\nu,j)}_{a,m}\mathbf J\mathbf R\mathbf V_{1,\nu-1}\mathbf e_j\mathbf e_j^T
 \mathbf V_{\nu,m}
\mathbf R\mathbf V_{1,b}]_{k,k+p_{m-b}}|\notag\\&\qquad\le [\mathbf
J\mathbf H^{(\alpha+1)}\mathbf V_{\alpha+2,m-\alpha}\mathbf
H^{(m-\alpha+1,j)} \mathbf V_{m-\alpha+2,m}\mathbf R\mathbf
V_{1,m-\alpha}\mathbf V^{(j)}_{\alpha+1,m}\mathbf J\mathbf R\mathbf
V_{1,\alpha}]_{jj} .\notag
\end{align}
This implies that
\begin{equation}\notag
 T^{(2)}\le C\E\|[\mathbf V_{\nu+1,m}\mathbf J\mathbf R\mathbf V_{1,b}\mathbf V_{a,m}\mathbf J\mathbf R\mathbf
\mathbf V_{1,\nu}\|_2^2.
\end{equation}
It is straightforward to check
\begin{equation}\label{t2}
 T^{(2)}\le \frac C{v^4}\E\|\mathbf V_{1,\alpha}\mathbf J\mathbf H^{(\alpha+1)}\mathbf V_{\alpha+2,m-\alpha}\mathbf H^{(m-\alpha+1,j)}
\mathbf V_{m-\alpha+2,m}\|_2^2=\E\|\mathbf Q\|_2^2
\end{equation}
The matrix on  the right hand side of equation (\ref{t2}) may be represented
 in the form
\begin{equation}\notag
 Q=\prod_{\nu=1}^m{\mathbf H^{(\nu)}}^{\varkappa_{\nu}},
\end{equation}
where $\varkappa_{\nu}=0$ or $\varkappa_{\nu}=1$ or $\varkappa_{\nu}=2$.
Since $X^{(\nu)}_{ss}=0$, for $\varkappa=1$ or $\varkappa=2$, we have
\begin{equation}\notag
 \E|{\mathbf H^{(\nu)}}^{\varkappa}_{kl}|^2\le \frac C{n}.
\end{equation}
This implies that
\begin{equation}\label{T2}
 T_2\le Cn.
\end{equation}
Similar we prove that
\begin{equation}\label{T3}
 T_3:=\sum_{j=1}^n\E|\Xi_j^{(3)}|^2\le Cn.
\end{equation}
Inequality (\ref{T1}), (\ref{T2}) and (\ref{T3}) together imply
\begin{equation}\notag
 \sum_{j=1}^n\E|\Xi_j|^2\le Cn
\end{equation}
Applying now a martingale expansion with respect to the 
$\sigma$-algebras $\mathcal F_j$ generated the random variables
$X_{kl}^{(\alpha+1)}$ with $1\le k\le j$, $1\le l\le n$ and all other random variables $X^{(q)}_{sl}$
except $q=\alpha+1$, we get
\begin{equation}\notag
\E|\frac1n(\sum_{k =1}^n[\mathbf V_{\alpha+1,m}\mathbf J\mathbf R\mathbf  V_{1,m-\alpha}]_{kk+n}-\E\sum_{j=1}^n[\mathbf V_{\alpha+1,m}\mathbf J\mathbf R\mathbf V_{1,m-\alpha}]_{kk+n})|^2\le \frac C{n v^4}.
\end{equation}

Thus the Lemma is proved.

\end{proof}

\begin{lem}\label{derivatives}
Under the conditions of Theorem \ref{main} we have, for $\alpha=1,\ldots,m,$
that  there exists a constant  $C$  such that
\begin{equation}\notag
\frac1{n^{\frac32}}\E\left|\sum_{j=1}^{p_{\alpha-1}}\sum_{k=1}^{p_{\alpha}}(-X^{(\alpha)}_{jk}+(1-\theta_{jk}){X^{(\alpha)}_{jk}}^3)\left[\frac{\partial ^{2}(\mathbf V_{\alpha+1,m}\mathbf J\mathbf R\mathbf V_{1,m-\alpha+1})}{\partial {X_{jk}^{(\alpha)}}^{2}}(\theta_{jk}^{(\alpha)}X_{jk}^{(\alpha)})\right]_{kj}\right|
\le C\tau_nv^{-4},
\end{equation}
and
\begin{align}\notag
\frac1{n^{\frac32}}&\E\left|\sum_{j=1}^{p_{m-\alpha}}\sum_{k=1}^{p_{m-\alpha+1}}(-X^{(m-\alpha+1)}_{jk}+(1-\theta_{jk}){X^{(m-\alpha+1)}_{jk}}^3)\right.\notag\\ &\qquad\qquad\left.\times\left[\frac{\partial ^{2}(\mathbf V_{\alpha+1,m}\mathbf J\mathbf R\mathbf V_{1,m-\alpha+1})}{\partial {X_{jk}^{(m-\alpha+1)}}^{2}}(\theta_{jk}^{(m-\alpha+1)}X_{jk}^{(m-\alpha+1)})\right]_{j+p_{\alpha-1},k}\right|
\le C\tau_nv^{-4},
\end{align}
where $\theta_{jk}^{(\alpha)}$ and $X_{jk}^{(\alpha)}$ are r.v. which are 
 independent in aggregate for
$\alpha=1,\ldots,m$ and $j=1,\ldots,p_{\alpha-1}$, $k=1,\ldots,p_{\alpha}$, and
$\theta_{jk}^{(\alpha)}$ are uniformly distributed on the unit interval.\newline
By $\frac{\partial^2}{\partial {X_{jk}^{(\alpha)}}^2}\mathbf
A(\theta_{jk}^{(\alpha)}X_{jk}^{(\alpha)})$ we denote the  matrix obtained from
$\frac{\partial^2}{\partial {X_{jk}^{(\alpha)}}^2}\mathbf A$ by replacing its entries 
$X_{jk}^{(\alpha)}$ by $\theta_{jk}^{(\alpha)}X_{jk}^{(\alpha)}$.
\end{lem}
\begin{proof}The proof of this lemma is rather technical. But  for
  completeness we shall include it here.
By {the formula for the derivatives of a resolvent matrix}, we have
\begin{equation}\label{deriv}
\frac{\partial (\mathbf V_{\alpha+1,m}\mathbf J\mathbf R\mathbf V_{1,m-\alpha+1})}{\partial X_{jk}^{(\alpha)}}=
\sum_{l=1}^5Q_l,
\end{equation}
\begin{align}
\mathbf Q_1=&\frac1{\sqrt n}
\mathbf V_{\alpha+1,m}\mathbf J\mathbf R\mathbf V_{1,\alpha-1}\mathbf e_j\mathbf e_k^T\mathbf V_{\alpha+1,m-\alpha+1}I_{\{\alpha\le m-\alpha+1\}})\notag\\
\mathbf Q_2=&\frac1{\sqrt n}\mathbf V_{\alpha+1,m}\mathbf J\mathbf R\mathbf V_{1,m-\alpha}\mathbf e_{k+p_{m-\alpha}}\mathbf e_{j+p_{m-\alpha+1}}\notag\\
\mathbf Q_3=&-\frac1{\sqrt n}\mathbf V_{\alpha+1,m}\mathbf J\mathbf R\mathbf V_{1,\alpha-1}\mathbf e_j\mathbf e_k^T\mathbf V_{\alpha+1,m}\mathbf J\mathbf R\mathbf V_{1,m-\alpha+1}\notag\\
\mathbf Q_4=&-\frac1{\sqrt n}\mathbf V_{\alpha+1,m}\mathbf J\mathbf R\mathbf V_{1,m-\alpha}\mathbf e_{k+p_{m-\alpha}}\mathbf e_{j+p_{m-\alpha+1}}^T\mathbf V_{m-\alpha+2,m}\mathbf J\mathbf R\mathbf V_{1,m-\alpha+1}\notag\\
\mathbf Q_5=&\frac1{\sqrt n}\mathbf V_{\alpha+1,m-\alpha}\mathbf e_{k+p_{m-\alpha}}\mathbf e_{j+p_{m-\alpha+1}}^T
\mathbf V_{m-\alpha+2,m}\mathbf J\mathbf R\mathbf V_{1,m-\alpha+1}I_{\{\alpha\le m-\alpha+1\}}).\notag
\end{align}
Introduce the notations
\begin{equation}\notag
 \mathbf U_{\alpha}:=\mathbf V_{\alpha+1,m},\quad\mathbf V_{\alpha}=\mathbf V_{1,m-\alpha+1}.
\end{equation}
From  formula (\ref{deriv}) it follows that
\begin{equation}\notag
 \frac{\partial^2 (\mathbf U_{\alpha}\mathbf J\mathbf R\mathbf V_{\alpha})}{\partial {X_{jk}^{(\nu)}}^2}=\sum_{l=1}^{5}\frac{\partial \mathbf Q_l}{\partial X_{jk}^{(\alpha)}}.
\end{equation}
Since all other calculations will be  similar we consider the case $l=3$ only.
Simple calculations show that
\begin{equation}
  \frac{\partial \mathbf Q_3}{\partial X_{jk}^{(\alpha)}}=\sum_{m=1}^{7} \mathbf P^{(m)},
\end{equation}
 where
\begin{align}
\mathbf P^{(1)}&=-\frac1n\mathbf V_{\alpha+1,m-\alpha}\mathbf e_{k+p_{m-\alpha}}\mathbf e_{j+p_{m-\alpha+1}}^T
\mathbf U_{m-\alpha+1}\mathbf J\mathbf R\mathbf V_{m-\alpha+2}\mathbf e_{j}\mathbf e_{k}^T\mathbf U_{\alpha}\mathbf J\mathbf R\mathbf V_{\alpha}\notag\\
\mathbf P^{(2)}&=-\frac1n\mathbf U_{\alpha}\mathbf J\mathbf R\mathbf V_{m-\alpha+2}\mathbf e_{j}\mathbf e_{k}^T\mathbf U_{\alpha}\mathbf J\mathbf R\mathbf V_{\alpha+1}\mathbf e_{k+p_{m-\alpha}}\mathbf e_{j+p_{m-\alpha+1}}^T
\notag\\
\mathbf P^{(3)}&=-\frac1n\mathbf U_{\alpha}\mathbf J\mathbf R\mathbf V_{m-\alpha+2}\mathbf e_{j}\mathbf e_{k}^T\mathbf V_{\alpha+1,m-\alpha}\mathbf e_{k+p_{m-\alpha}}\mathbf e_{j+p_{m-\alpha+1}}^T\mathbf U_{m-\alpha+1}\mathbf J\mathbf R\mathbf V_{\alpha}\notag\\
\mathbf P^{(4)}&=-\frac1n\mathbf U_{\alpha}\mathbf J\mathbf R\mathbf V_{m-\alpha+2}\mathbf e_j\mathbf e_k^T\mathbf U_{\alpha}\mathbf J\mathbf R\mathbf V_{m-\alpha+2}\mathbf e_j\mathbf e_k^T\mathbf U_{\alpha}\mathbf J\mathbf R\mathbf V_{\alpha}\notag\\
\mathbf P^{(5)}&=\frac1n\mathbf U_{\alpha}\mathbf J\mathbf R\mathbf V_{\alpha+1}\mathbf e_{k+p_{m-\alpha}}\mathbf e_{j+p_{m-\alpha+1}}^T\mathbf U_{m-\alpha+1}\mathbf J\mathbf R\mathbf V_{m-\alpha+2}\mathbf e_j\mathbf e_k^T\mathbf U_{\alpha}\mathbf J\mathbf R\mathbf V_{\alpha}
\notag\\
\mathbf P^{(6)}&=\frac1n\mathbf U_{\alpha}\mathbf J\mathbf R\mathbf V_{m-\alpha+2}\mathbf e_j\mathbf e_k^T\mathbf U_{\alpha}\mathbf J\mathbf R\mathbf V_{\alpha+1}\mathbf e_{k+p_{m-\alpha}}\mathbf e_{j+p_{m-\alpha+1}}^T\mathbf U_{m-\alpha+1}\mathbf J\mathbf R\mathbf V_{\alpha}
\notag\\
\mathbf P^{(7)}&=\frac1n\mathbf U_{\alpha}\mathbf J\mathbf R\mathbf V_{m-\alpha+2}\mathbf e_j\mathbf e_k^T\mathbf U_{\alpha}\mathbf J\mathbf R\mathbf V_{m-\alpha+2}\mathbf e_{j}\mathbf e_{k}^T\mathbf U_{\alpha}\mathbf J\mathbf R\mathbf V_{\alpha}.\notag
\end{align}

Consider now the quantity, for $\mu=1,\ldots,5$,
\begin{equation}\label{fin1}
 L_{\mu}=\frac1{n^{\frac32}}\sum_{j=1}^{p_{\alpha-1}}\sum_{k=1}^{p_{\alpha}}\E {X_{j,k}^{(\alpha)}}^3
\left[\frac{\partial \mathbf Q_{\mu}}{\partial X_{jk}^{(\alpha)}}\right]_{kj}.
\end{equation}
We bound $L_3$ only. The others bounds  are similar.
First we note that
\begin{equation}\label{fin2}
 \sum_{j=1}^{p_{\alpha-1}}\sum_{k=1}^{p_{\alpha}}\E {X_{j,k}^{(\alpha)}}^3[\mathbf P^{(\nu)}]_{kj}=0,\quad\text{for}\quad \nu=1,2,3.
\end{equation}
Furthermore,
\begin{equation}
 \E |{X_{j,k}^{(\alpha)}}^3||[\mathbf P^{(4)}]_{kj}|\le\E|X_{jk}^{(\alpha)}|^3
|[\mathbf U_{\alpha}\mathbf J\mathbf R\mathbf V_{m-\alpha+2}]_{kj}|^2|[\mathbf U_{\alpha}\mathbf J\mathbf R
\mathbf V_{\alpha}]_{kj}|.
\end{equation}

Let $\mathbf U_{\alpha}^{(jk)}$ ( $\mathbf V^{(j,k)}_{\alpha}$) denote matrix obtained from $\mathbf U_{\alpha}$  ($\mathbf V_{\alpha}$) by replacing $X_{jk}^{(\alpha)}$ by zero. We may write
\begin{align}
\mathbf U_{\alpha}= \mathbf U_{\alpha}^{(jk)}+\frac1{\sqrt n}X_{jk}^{(\alpha)}\mathbf V_{\alpha+1,m-\alpha+1}\mathbf e_{k+p_{m-\alpha}}\mathbf e_{j+p_{m-\alpha+1}}^T\mathbf V_{m-\alpha+2,m}.
\end{align}
and
\begin{align}
 \mathbf V_{\alpha}=\mathbf V^{(j,k)}_{\alpha}+\frac1{\sqrt n}X_{jk}\mathbf V_{1,m-\alpha+1}\mathbf e_{k+p_{m-\alpha}}\mathbf e_{j+p_{m-\alpha+1}}^T.\notag
\end{align}
 Using these representations and taking in account that
\begin{equation}
 [\mathbf V_{\alpha+1,m-\alpha}]_{k,k+p_{m-\alpha}}=[\mathbf V_{1,m-\alpha}]_{k,k+p_{m-\alpha}}=0,
\end{equation}
 we get
\begin{align}\label{fin8}
 \E |{X_{j,k}^{(\alpha)}}^3||[\mathbf P^{(4)}]_{kj}|\le \frac1{n}\E |{X_{j,k}^{(\alpha)}}^3||[\mathbf U_{\alpha}\mathbf J\mathbf R\mathbf V_{m-\alpha+2}]_{kj}|^2
|[\mathbf U^{(j,k)}_{\alpha}\mathbf J\mathbf R\mathbf V^{(j,k)}_{\alpha}]_{kj}|.
\end{align}
Furthermore,
\begin{align}\label{fin9}
|[\mathbf U_{\alpha}\mathbf J\mathbf R\mathbf V_{m-\alpha+2}]_{k,j}|&\le
\frac1v\|\mathbf V_{m-\alpha+2}\mathbf e_{j}\|_2\|\mathbf e_k^T\mathbf U_{\alpha}\|_2\notag\\
|[\mathbf U^{(j,k)}_{\alpha}\mathbf J\mathbf R\mathbf V^{(j,k)}_{\alpha}]_{kj}|&\le\frac1v
\|\mathbf V^{(j,k)}_{\alpha}\mathbf e_{k}\|_2\|\mathbf e_j^T\mathbf U^{(j,k)}_{\alpha}\|_2.\notag\\
\end{align}
Applying inequalities (\ref{fin8}) and (\ref{fin9}) and taking in account 
the independence of entries, we get
\begin{align}
\E |{X_{j,k}^{(\alpha)}}^3||[\mathbf P^{(4)}]_{kj}|\le \frac1{nv^2}\E |{X_{j,k}^{(\alpha)}}^3 \E\|\mathbf V_{m-\alpha+2}\mathbf e_{k}\|_2^2\|\mathbf e_j^T\mathbf U_{\alpha}\|_2^2\|\mathbf V^{(j,k)}_{\alpha}\mathbf e_{k}\|_2\|\mathbf e_j^T\mathbf U^{(j,k)}_{\alpha}\|_2
\end{align}
Applying Lemma \ref{norm4}, we get
\begin{equation}
 \frac1{n^{\frac32}}\sum_{j=1}^{p_{\alpha-1}}\sum_{k=1}^{p_{\alpha}}\E |X_{jk}^{(\alpha)}|^3
|[\mathbf P^{(4)}]_{kj}|\le \frac{C}{n^{\frac52}}\sum_{j=1}^{p_{\alpha-1}}\sum_{k=1}^{p_{\alpha}}\E |X_{jk}^{(\alpha)}|^3
\end{equation}
Assumption (\ref{as1}) now  yields
\begin{equation}
 \frac1{n^{\frac32}}\sum_{j=1}^{p_{\alpha-1}}\sum_{k=1}^{p_{\alpha}}\E |X_{jk}^{(\alpha)}|^3
|[\mathbf P^{(4)}]_{kj}|\le C\tau_n.
\end{equation}
Similar we get the bounds for  $\nu=5,6,7$
\begin{equation}
 \frac1{n^{\frac32}}\sum_{j=1}^{p_{\alpha-1}}\sum_{k=1}^{p_{\alpha}}\E |X_{jk}^{(\alpha)}|^3
|[\mathbf P^{(\nu)}]_{kj}|\le C\tau_n.
\end{equation}
and
\begin{equation}
 |L_{\mu}|\le C\tau_n,\quad \mu=1,\ldots,5.
\end{equation}

The bound of the quantity 
\begin{equation}\label{fin1*}
 \widehat L_{\mu}=\sum_{j=1}^{p_{\alpha-1}}\sum_{k=1}^{p_{\alpha}}\E {X_{j,k}^{(\alpha)}}
\left[\frac{\partial \mathbf Q_{\nu}}{\partial X_{jk}^{(\alpha)}}\right]_{kj}.
\end{equation}
is similar.
Thus, the Lemma is proved.

\end{proof}
\begin{lem}\label{teilor}Under the  conditions of Theorem \ref{main} we have
\begin{equation}\notag
\sum_{j=1}^{p_{\nu-1}}\sum_{k=1}^{p_{\nu}}\E X_{jk}^{(\nu)}[\mathbf V_{\nu+1,m}\mathbf J\mathbf R\mathbf V_{1,m-\nu+1}]_{kj}=\sum_{j=1}^{p_{\nu-1}}\sum_{k=1}^{p_{\nu}}\E\left[\frac{\partial \mathbf V_{\nu+1,m}\mathbf J\mathbf R\mathbf V_{1,m-\nu+1}}{\partial X_{jk}^{(\nu)}}\right]_{kj}+\varepsilon_n(z)
\end{equation}
and
\begin{align}
\sum_{j=1}^{p_{m-\nu}}\sum_{k=1}^{p_{m-\nu+1}}&\E X^{(m-\nu+1)}_{j,k}[\mathbf V_{\nu+1,m}\mathbf J\mathbf R\mathbf V_{1,m-\nu+1}]_{j+p_{\nu-1},k}\notag\\&=\sum_{j=1}^{p_{m-\nu}}\sum_{k=1}^{p_{m-\nu+1}}\E\left[\frac{\partial \mathbf V_{\nu+1,m}\mathbf J\mathbf R\mathbf V_{1,m-\nu+1}}{\partial X_{jk}^{(\nu)}}\right]_{j+p_{\nu-1},k}+\varepsilon_n(z),\notag
\end{align}
where $|\varepsilon_n(z)|\le\frac{C\tau_n}{v^4}$.
\end{lem}
\begin{proof}
 We apply Taylor's formula twice,
\begin{equation}\notag
 \E\xi f(\xi)=f'(0)\E\xi^2+\E\xi^3f''(\theta\xi)(1-\theta),
\end{equation}
and
\begin{equation}
  f'(0)=\E f'(\xi)-\E\xi f''(\theta\xi)
\end{equation}
where $\theta$ denotes  uniformly  distributed r.v. on the unit interval 
which is  independent of $\xi$.
After simple calculations we get
\begin{align}
 \sum_{j=1}^{p_{\nu-1}}\sum_{k=1}^{p_{\nu}}\E X_{jk}^{(\nu)}&[\mathbf V_{\nu+1,m}\mathbf J\mathbf R\mathbf V_{1,m-\nu+1}]_{kj}=
\sum_{j=1}^{p_{\nu-1}}\sum_{k=1}^{p_{\nu}}\E\left[\frac{\partial \mathbf V_{\nu+1,m}\mathbf J\mathbf R\mathbf V_{1,m-\nu+1}}{\partial X_{jk}^{(\nu)}}\right]_{kj}\notag\\
&+\sum_{j=1}^{p_{\nu-1}}\sum_{k=1}^{p_{\nu}}\E(-X_{jk}^{(\nu)}+(1-\theta_{jk}){X_{jk}^{(\nu)}}^3)\left[\frac{\partial^2 \mathbf V_{\nu+1,m}\mathbf J\mathbf R\mathbf V_{1,m-\nu+1}}{{\partial X_{jk}^{(\nu)}}^2}(\theta_{jk}^{(\nu)}
X_{jk}^{(\nu)})\right]_{kj}.\notag
\end{align}
Using the results of Lemma \ref{derivatives}, we conclude the proof.

\end{proof}



\end{document}